\documentclass[12pt]{article}
\usepackage{marktext} 

\usepackage{svmacro}
\usepackage[utf8]{inputenc}
\usepackage{authblk}
\usepackage[style=alphabetic, backend=biber,maxbibnames=99 ]{biblatex}
\makeatletter
\newcommand{\subjclass}[2][2010]{%
  \let\@oldtitle\@title%
  \gdef\@title{\@oldtitle\footnotetext{#1 \emph{Mathematics subject classification.} #2}}%
}
\newcommand{\keywords}[1]{%
  \let\@@oldtitle\@title%
  \gdef\@title{\@@oldtitle\footnotetext{\emph{Key words and phrases.} #1.}}%
}
\makeatother

\title{Coagulation-Fragmentation equations with multiplicative coagulation kernel and constant fragmentation kernel}

\begin{document}

\author[1]{Hung V. Tran}
\author[2]{Truong-Son Van}
\affil[1]{Department of Mathematics,
University of Wisconsin Madison,
Van Vleck Hall, 480 Lincoln Drive, Madison, WI 53706}
\affil[2]{Department of Mathematical Sciences and
Center for Nonlinear Analysis,
Carnegie Mellon University, Pittsburgh, PA 15213}

\date{\today}



\maketitle
\begin{abstract}
    We study a critical case of Coagulation-Fragmentation equations with multiplicative coagulation kernel and constant fragmentation kernel.
    Our method is based on the study of viscosity solutions to a new singular Hamilton-Jacobi equation, which results from applying the Bernstein transform to the original Coagulation-Fragmentation equation.
    Our results include wellposedness, regularity, and long-time behaviors of viscosity solutions to the Hamilton-Jacobi equation in certain regimes, which have implications to wellposedness and long-time behaviors of \emph{mass-conserving} solutions to the Coagulation-Fragmentation equation.
\end{abstract}

\tableofcontents

\section{Introduction}
The Coagulation-Fragmentation equation (C-F) is an integrodifferential equation that finds applications in many different fields, ranging from astronomy to polymerization to the study of animal group sizes. 
The equation, with pure coagulation, dates back to Smoluchowski~\cite{Smoluchowski16}, when he studied the evolution of number density of particles as they coagulate.
Later on, Blatz and Tobolsky~\cite{BlatzTobolsky45} use the full C-F to study polymerization-depolymerization phenomena.
The mathematical studies of this equation did not start until the work of Melzak~\cite{Melzak57}, which was concerned with existence and uniqueness of the solutions for bounded kernels. 
Since then, although there are still a lot of open questions remained, major advancements have been made by both analytic and probabilistic tools.
We list here some, but not exhaustive, important works that are relevant to our work.
For existence and uniqueness of solutions, there are the works of McLeod~\cite{McLeod62}, Ball and Carr~\cite{BallCarr90}, Norris~\cite{Norris99}, Escobedo, Lauren\c{c}ot, Mischler and Perthame~\cite{EscobedoMischlerPerthame2002, EscobedoLaurencotEA03}.  
For large time behavior of solutions, there are the works of Aizenman and Bak~\cite{AizenmanBak79}, Ca\~{n}izo~\cite{Canizo07}, Carr~\cite{Carr92}, Menon and Pego~\cite{MenonPego04, MenonPego06, MenonPego08}, Degond, Liu and Pego~\cite{DegondLiuEA17}, Liu, Niethammer and Pego~\cite{LiuNiethammerPego19}, Niethammer and Vel\'{a}zquez~\cite{NiethammerVelazquez13} and Lauren\c{c}ot~\cite{LAURENCOT2019}.
For surveys of what has been done, we refer the readers to two dated by now but still excellent surveys by Aldous~\cite{Aldous99} and da Costa~\cite{Costa15} and the new monographs by Banasiak, Lamb, and Lauren\c{c}ot~\cite{BanasiakLambEA19}.  

Here, coagulation represents binary merging when two clusters of particles meet, which happens at some pre-determined rates; and fragmentation represents binary splitting of a cluster, also at some pre-determined rates. 
Thus, the C-F describes the evolution of cluster sizes over time given that there are only coagulation and fragmentation that govern the dynamics.

A particularly interesting phenomenon of the C-F is that
 given the right conditions, the solution, while still physical, does not conserve mass at all time.
 There are two ways that this could happen. One comes from the formation of particles of infinite size; the other comes from the formation of particles of size zero, both in finite time. 
 The first, called \emph{gelation}, happens when the coagulation is strong enough~\cite{EscobedoLaurencotEA03}.
 The latter, called \emph{dust formation}, happens when the fragmentation is strong enough (see Bertoin~\cite{Bertoin06}). 
 Typically, these phenomena happen depending on the relative strengths between the coagulation kernel and fragmentation kernel, not so much on the initial data. However, there are borderline situations, where it is not very clear how solutions would behave, hence more careful analysis needs to be done based on initial data.
 Both are very interesting and rich phenomena, and have been studied in various contexts.

  The main goal of this article is to propose a new framework to analyze a borderline situation described by Escobedo, Lauren\c{c}ot, Mischler and Perthame~\cite{EscobedoMischlerPerthame2002, EscobedoLaurencotEA03}, where solutions to the C-F may or may not exhibit gelation, depending on the initial data (as opposed to the type of kernels). 
 In particular, we analyze the properties of viscosity solutions of a new singular Hamilton-Jacobi equation (H-J), which results from transforming the C-F equation via the so-called Bernstein transform.
 This, in our opinion, is natural and elegant since it requires very minimal assumptions.

 We note that, the Bernstein transform was first used to analyze this type of equations by Menon and Pego~\cite{MenonPego04} (under the name ``desingularized Laplace transform''). This transform is a generalization of the Laplace transform and has properties that fit well with properties of solutions of C-F.

 \subsection{The Coagulation-Fragmentation equation}
 Let $c(s,t) \geq 0$ be the density of clusters of particles of size $s \geq 0$ at time $t\geq 0$.
 We write the (continuous) Coagulation-Fragmentation equation as following
\begin{equation}\label{eq:CF}
\partial_t c(s,t) = Q_c(c) + Q_f(c)\,.
\end{equation}
Here, the coagulation term  $Q_c$ and the fragmentation term $Q_f$ are given by
\[
Q_c(c)(s,t)=\frac{1}{2} \int_0^s a(y,s-y)c(y,t) c(s-y,t)\,dy - c(s,t) \int_0^\infty a(s,y) c(y,t)\,dy\,,
\]
and
\[
Q_f(c)(s,t)=-\frac{1}{2} c(s,t) \int_0^s b(s-y,y)\,dy + \int_0^\infty b(s,y) c(y+s,t)\,dy\,.
\]
In the above, $a, b$ denote the coagulation kernel and fragmentation kernel, respectively, which are nonnegative and symmetric functions defined on $(0,\infty)^2$.

Throughout this paper, we always assume that
\[
a(s,\hat s) = s\hat s \quad \text{ and } \quad
b(s,\hat s) = 1
\quad \text{ for all } s, \hat s>0\,.
\]

\subsection{The Bernstein transform}
We take a weak form of the coagulation-fragmentation equation. 
We say that $c(s,t)$ is a solution to the coagulation-fragmentation equation \eqref{eq:CF} if
for every test function $\phi \in BC([0,\infty)) \cap \rm{Lip}([0,\infty))$ with $\phi(0)=0$,
we have
\begin{equation}\label{weak sln}
  \begin{split}
    \frac{d}{dt} \int_0^\infty \phi(s) c(s,t) \, ds &= \frac{1}{2} \int_0^\infty \int_0^\infty (\phi(s + \hat s) - \phi(s) - \phi(\hat s)) s c(s,t) \hat s c(\hat s,t) \, d\hat s ds \\ 
    & \quad -\frac{1}{2} \int_0^\infty \paren[\Big]{ \int_0^s (\phi(s) - \phi(\hat s) - \phi(s-\hat s)) \,d\hat s  } c(s,t) \, ds \,.
  \end{split}
\end{equation}
Here, $BC([0,\infty))$ is the class of bounded continuous functions on $[0,\infty)$, and $\rm{Lip}([0,\infty))$ is the class of Lipschitz continuous functions on $[0,\infty)$.
Consider the Bernstein transform of $c$, for $(x,t)\in [0,\infty)^2$,
\begin{equation*}
  F(x,t) \defeq \int_0^\infty (1 - e^{-sx}) c(s,t) \, ds 
\end{equation*}
and let
\begin{equation*}
  \phi_x (s) = 1 - e^{-sx}  \,,
\end{equation*}
we have

\begin{align*}
  \partial_t F(x,t) &= \frac{1}{2} \int_0^\infty \int_0^\infty ( 1 - e^{-(s + \hat s) x} - 1 + e^{-sx} -1 + e^{-\hat s x} ) s c(s,t) \hat c(\hat s, t) \, d\hat s ds \\
  & - \frac{1}{2} \int_0^\infty \int_0^s ( 1 - e^{-sx} - 1 + e^{-(s - \hat s)x} - 1 + e^{- \hat s x} ) \, d \hat s \, c(s,t) \, ds \\ 
  &=  -\frac{1}{2} \int_0^\infty \int_0^\infty (1 - e^{-sx})(1 - e^{-\hat s x}) s c(s,t) \hat s c(\hat s,t) \, d\hat s ds \\
  & - \frac{1}{2} \int_0^\infty ( - s - s e^{-sx} + \frac{2}{x}(1  - e^{-sx}) ) c(s,t) \, ds  \\
  &= - \frac{1}{2}( m_1(t) - \partial_x F(x,t))^2 + \frac{m_1(t)}{2} + \frac{ \partial_x F(x,t)}{2} - \frac{F(x,t)}{x} \\
  &= - \frac{1}{2} ( m_1(t) - \partial_x F(x,t)) ( m_1(t) - \partial_x F(x,t) + 1) - \frac{F(x,t)}{x} + m_1(t) \,.
\end{align*}
Here, $m_1(t)$ is the total mass (first moment) of all particles at time $t \geq 0$, that is,
\[
m_1(t) = \int_0^\infty sc(s,t)\,ds\,.
\]
Let us assume that $m_1(t)<\infty$ for all $t \geq 0$.
The key point is to transform a seemingly hopeless nonlocal equation to a somewhat more tractable nonlinear PDE, which enjoys
some major developments in the past few decades.
If conservation of mass holds, then we can assume $m_1(t)=m >0$ for all $t\geq 0$ for some $m\in (0,\infty)$.
This fact, together with the above computations, leads to the following PDE for $F$.
\begin{subequations} \label{e:main}  
  \begin{gather}
    \partial_t F  + \frac{1}{2}(\partial_x F - m)(\partial_x F - m -1) + \frac{F}{x} - m = 0 \quad \text{ in } (0,\infty)^2 \,, \\
    0 \leq F(x,t) \leq mx \quad \text{ on } [0,\infty)^2 \,, \label{ine:sublinear} \\
    F(x,0) = F_0(x) \quad \text{ on } [0,\infty) \,. 
  \end{gather}
\end{subequations}
One then can study wellposedness and properties of solutions of \eqref{e:main} to deduce back information of C-F. Indeed,
this is our main goal.

Note that the condition~\eqref{ine:sublinear} implies that $F(0,t) = 0$ and that it comes directly from the Bernstein transform. 
Indeed, as $c \geq 0$, it is clear that $F \geq 0$.
Besides, the inequality $1- e^{-sx} \leq sx$ for $s,x \geq 0$ gives
\[
F(x,t) = \int_0^\infty (1 - e^{-sx}) c(s,t) \, ds \leq \int_0^\infty sx c(s,t) \, ds =mx\, .
\]
Moreover, the dominated convergence theorem gives
\[
\lim_{x\to \infty} \frac{F(x,t)}{x} = \lim_{x\to \infty} \int_0^\infty \frac{1 - e^{-sx}}{x} c(s,t) \, ds =0\, ,
\]
which means that $F(x,t)$ is sublinear in $x$.
Here, for a given function $\psi:[0,\infty) \to \R$, we say that it is sublinear if
\[
\lim_{x\to \infty} \frac{\psi(x)}{x} = 0.
\]
It is therefore natural to search for solutions of \eqref{e:main} that are sublinear in $x$.

\medskip

It is worth noting that \eqref{e:main} is a Hamilton-Jacobi equation with the Hamiltonian
\[
H(p,z,x) = \frac{1}{2}(p-m)(p-m-1) + \frac{z}{x} -m \quad \text{ for all } (p,z,x) \in \R \times \R \times (0,\infty)\, ,
\]
which is of course singular at $x=0$. 
Besides, $H$ is monotone, but not Lipschitz in $z$ as
\[
\partial_z H(p,z,x)=\frac{1}{x} \geq 0 \quad \text{ and } \quad
\lim_{x \to 0+} \partial_z H(p,z,x) = \lim_{x \to 0+} \frac{1}{x} = +\infty\,.
\]
This means that \eqref{e:main} does not fall into the classical theory of viscosity solutions to Hamilton-Jacobi equations developed by Crandall and Lions \cite{CrandallLions83} (see also Crandall, Evans and Lions \cite{CrandallEvansLions84}).
It is thus our purpose to develop a framework to study wellposedness and further properties of solutions to \eqref{e:main}.
For a different class of Hamilton-Jacobi equations that is singular in $p$ (but not in $z$), see the radially symmetric setting in Giga, Mitake and Tran \cite{GigaMitakeTran17}.

We emphasize that for wellposedness and regularity results, we do not need to impose all the properties of the Bernstein transform of the initial data $c_0=c(\cdot,0)$. To be precise, a Bernstein transform of a measure is a $C^\infty((0,\infty))$ (in fact, analytic) function. However, we only assume $F_0$ to be Lipschitz and sublinear for our wellposedness result and more regular for our regularity results.

A more important point is that our assumption on $c_0$ is minimal. For existence and uniqueness results, we do not have any restrictions on moments of $c_0$ except finite first moment so that the derivative of the Bernstein transform makes sense. 
In particular, we  only require
\begin{equation*}
    m_1(0) = \int_0^\infty s c_0(s) \, ds <\infty \,.
\end{equation*}
This also makes physical sense since one often wishes that the initial total mass to be finite before talking about conservation of mass.
Of course, we will need to put in more conditions for our regularity results.

\begin{remark}
In fact, we are also able to define weak solutions in the measure sense to \eqref{eq:CF} in a similar fashion.

\medskip

For each $t \geq 0$, let $c_t(ds)$ be a positive Radon measure in $(0,\infty)$.
Then, we say that $c_t(ds)$ is a weak solution in the measure sense to \eqref{eq:CF} if for every test function $\phi \in BC([0,\infty)) \cap \rm{Lip}([0,\infty))$ with $\phi(0)=0$,
we have
\begin{equation*}
  \begin{split}
    \frac{d}{dt} \int_0^\infty \phi(s) \, c_t(ds) &= \frac{1}{2} \int_0^\infty \int_0^\infty (\phi(s + \hat s) - \phi(s) - \phi(\hat s)) s \hat s  \, c_t(ds)c_t(d\hat s) \\ 
    & \quad -\frac{1}{2} \int_0^\infty \paren[\Big]{ \int_0^s (\phi(s) - \phi(\hat s) - \phi(s-\hat s)) \,d\hat s  }  \, c_t(ds) \,.
  \end{split}
\end{equation*}
This is clearly a weaker notion of solutions than that in \eqref{weak sln}. Nevertheless, the Bernstein transform of $c_t(ds)$ and \eqref{e:main} still make perfect sense. We will use this notion of solutions when talking about the existence results for the C-F.
\end{remark}

\subsection{A conjecture}
In \cite{EscobedoMischlerPerthame2002, EscobedoLaurencotEA03}, the authors conjectured that in borderline situations where coagulation kernel and fragmentation kernel balance each other out, the solution will conserve mass if the initial data have small enough total mass. Otherwise, for large total mass initial data, gelation will occur. In the paper by Vigil and Ziff~\cite{VigilZiff89}, the authors argued that if the zeroth moment of the solution reaches negative value in finite time, one expects coagulation to dominate, hence gelation will occur. 

It has been expected by experts in the field that for our specific kernels, the critical initial mass should be $m_1(0) =1$ so that for $m_1(0) >1$, one has gelation; and for $m_1(0)\leq 1$, one has solutions that conserve mass. We give here a simple reason why such expectation arises.

Integrating equation~\eqref{eq:CF} and denoting $m_0(t) = \int_0^\infty c(s,t) \,ds$, the zeroth moment, we get the following equation

\begin{equation*}
    \frac{d}{dt} m_0(t) =  \frac{1}{2}m_1(t)( 1 - m_1(t)) \,.
\end{equation*}
Suppose now $m_1(t) = m_1(0) >1$ as it is true before gelation occurs (if ever). Then $m_0(t)$ will be negative in finite time.
On the other hand, $m_0(t)$ remains positive if $0\leq m_1(0) \leq 1$. Therefore, by the reasoning above, $m_1(0) = 1$ is believed to be the critical mass.
Our goal is to give  results towards resolving this conjecture, which will be detailed in the next subsection.

\subsection{Main results}
In this subsection, we give rigorous statements about our results, which we believe to be the stepping stones for further investigations in the future, both in the theory of viscosity solutions and in the theory of C-F.

First and foremost, we need to understand the existence and uniqueness of viscosity solutions for equation~\eqref{e:main}.

\begin{theorem}\label{thm:wellposedness} 
Assume that $0 <m \leq 1$.
Assume further that $F_0$ is Lipschitz, sublinear, and $ 0 \leq F_0(x) \leq mx$.
Then, \eqref{e:main} has a unique Lipschitz, sublinear solution $F$.
\end{theorem}

The proof of this theorem is given in Section~\ref{S:Wellposedness}.
Theorem \ref{thm:wellposedness} gives us a simple but important implication about C-F.
\begin{corollary}\label{cor:m leq 1}
 Assume that $m_1(0)=m \in (0,1]$.
 Then, equation~\eqref{eq:CF} has at most one mass-conserving solution.
\end{corollary} 

We believe that the uniqueness result of Corollary \ref{cor:m leq 1} is new in the literature although existence results of mass-conserving solutions for \eqref{eq:CF} for the whole range of $m_1(0) \in (0,1]$ are still not yet available.
In a recent important work, Lauren\c{c}ot \cite{LAURENCOT2019-2} showed existence and uniqueness of mass-conserving solutions to  \eqref{eq:CF} under some additional moment conditions for $0<m_1(0) < \frac{1}{4 \log 2}$.
In Theorem \ref{thm: existence of solution to C-F} below, we obtain existence (and of course uniqueness) of mass-conserving weak solutions in the measure sense to \eqref{eq:CF} in case that $0<m_1(0)<\frac{1}{2}$, and $c(\cdot,0)$ has bounded second moment.

 \medskip
 
We note that, in general, if the viscosity solution to the Hamilton-Jacobi equation forms shocks, one cannot have a solution of C-F that conserves mass anymore. This is because if there were a solution of C-F that conserves mass, its Bernstein transform would need to solve the Hamilton-Jacobi equation and at the same time would need to be smooth. This cannot be the case if there were shocks.

It is, therefore, of our interest to study the regularity of the viscosity solutions of the equation~\eqref{e:main}.
Moreover, regularity results in the theory of viscosity solutions are important in their own rights.

\begin{theorem}\label{thm:no C1 sln}
  Suppose $m >1$. 
  Assume that $F_0$ is smooth, sublinear, and $ 0 \leq F_0(x) \leq mx$.
  Then equation~\eqref{e:main} does NOT admit a solution  $F \in C^1([0,\infty)^2)$ which is sublinear in $x$.
\end{theorem}

The proof of this theorem is given in Subsection~\ref{S:nonExistence}.
Based on our discussion above, Theorem~\ref{thm:no C1 sln} implies immediately the following consequence.

\begin{corollary}\label{cor:m>1}
 Assume that $m_1(0)=m>1$.
 Then, there is no mass-conserving solution to equation~\eqref{eq:CF}.
\end{corollary} 
A version of Corollary \ref{cor:m>1} already appeared in \cite{BanasiakLambEA19}.
We here obtain non-existence of mass-conserving solutions under the minimal assumption, that is, $m_1(0)>1$.
We do not need to assume anything else about other moments. In particular, we do not need to impose that the zeroth moment, number of clusters, is finite as in \cite{BanasiakLambEA19}.
It is also worth noting that Corollaries \ref{cor:m leq 1} and \ref{cor:m>1} hold true for mass-conserving weak solutions in the measure sense to \eqref{eq:CF} as well.
\medskip

To study regularity of $F$ for $0<m \leq 1$, we impose more conditions on $F_0$ as following. 
Assume that there exist $\beta \in (0,1)$ and $C>0$ such that
\begin{gather}
  0 \leq F_0'(x) \leq m \quad \text{and} \quad F_0'(0)=m\,, \tag{A1} \label{A:firstDerivativeBound} \\
  -C \leq  F_0''(x)   \leq 0\,, \tag{A2} \label{A:secondDerivativeBound1} \\ 
 -\frac{m}{e} \leq x F_0''(x)  \leq 0 \quad \text{and} \quad \|xF_0''\|_{C^{0,\beta}([0,\infty))} \leq C \,. \tag{A3} \label{A:secondDerivativeBound2} 
\end{gather}

The above assumptions hold true when $F_0$ is the Bernstein transform of $c_0=c(\cdot,0)$, where $c_0$ has $m_1(0)=m$ and also bounded second moment, that is,
\[
m_2(0)=\int_0^\infty s^2 c(s,0)\,ds \leq C\,.
\]
Indeed,
\[
0\leq F_0'(x)=\int_0^\infty s e^{-xs}c(s,0)\,ds \leq m\, ,
\]
and $F_0'(0)=m$.
For second derivative, one has
\[
-C \leq F_0''(x)=-\int_0^\infty s^2 e^{-xs}c(s,0)\,ds \leq 0\,,
\]
and
\[
x F_0''(x)=-\int_0^\infty s^2x e^{-xs}c(s,0)\,ds =-\int_0^\infty (sx e^{-xs})sc(s,0)\,ds \geq -\frac{m}{e}\,.
\]
We use the fact that $re^{-r} \leq e^{-1}$ for $r\ge 0$ in the above.
Besides, for $x,y \in [0,\infty)$,
\begin{align*}
|x F_0''(x) - yF_0''(y)| &\leq \int_0^\infty |x e^{-xs}- y e^{-ys}|s^2c(s,0)\,ds \\
&\leq \int_0^\infty |x-y|s^2c(s,0)\,ds \leq C|x-y|.
\end{align*}
In the second inequality above, we use the point that $|x e^{-xs}- y e^{-ys}| \leq |x-y|$, which can be derived by the usual mean value theorem and 
\[
\left|(ze^{-zs})' \right| = \left| e^{-zs} - zs e^{-zs} \right| \leq 1 \qquad \text{ for all } z\geq 0.
\]

We first show that $F$ is always concave in $x$ provided that \eqref{A:firstDerivativeBound}--\eqref{A:secondDerivativeBound1} hold and $0 <m \leq 1$.

\begin{lemma} \label{lem:Fconcave}
  Assume {\rm \eqref{A:firstDerivativeBound}--\eqref{A:secondDerivativeBound1}}, and $0<m\leq 1$.
  Assume further that $F_0$ is sublinear, and $ 0 \leq F_0(x) \leq mx$.
  Then, the sublinear solution $F$ to the equation~\eqref{e:main} is concave in $x$ for each $t\geq 0$.
 \end{lemma}
 
 The concavity of $F$ in the above lemma is rather standard as the Hamiltonian is convex (in fact quadratic) in $p$.
 Of course, we need to be careful with the singularity of $H$ in $x$ at $x=0$, but otherwise, the arguments in the proof of Lemma \ref{lem:Fconcave} are quite classical.
 Next, we show that in a smaller range of $m$ ($0<m<\frac{1}{2}$), $F\in C^{1,1}((0,\infty)^2) \cap C^1([0,\infty)\times(0,\infty))$ under assumptions \eqref{A:firstDerivativeBound}--\eqref{A:secondDerivativeBound2}.
 It is worth noting that we do not need to put any assumption on third or higher derivatives of $F_0$.
 \begin{theorem}\label{thm:regularity-m-less-than-quarter}
   Assume {\rm \eqref{A:firstDerivativeBound}--\eqref{A:secondDerivativeBound2}}, and $0<m<\frac{1}{2}$.
  Assume further that $F_0$ is bounded, and $ 0 \leq F_0(x) \leq mx$.
  Then the sublinear solution $F$ to the equation~\eqref{e:main} is in $C^{1,1}((0,\infty)^2) \cap C^1([0,\infty)\times(0,\infty))$.
  Moreover, $F$ satisfies that, for $(x,t) \in (0,\infty)^2$,
  \[
  0 \leq \partial_x F(x,t) \leq m \quad \text{ and } \quad -1 \leq x \partial_x^2 F(x,t) \leq 0.
  \]
\end{theorem}

To the best of our knowledge, the regularity result in Theorem \ref{thm:regularity-m-less-than-quarter} is new in the literature.
The proofs of Lemma \ref{lem:Fconcave} and Theorem \ref{thm:regularity-m-less-than-quarter} are given in Subsection \ref{S:m between 0 and 1}.
Next is our existence result for C-F when $0<m<\frac{1}{2}$.

\begin{theorem}\label{thm: existence of solution to C-F}
Assume that $F_0$ is the Bernstein transform of $c_0=c(\cdot,0)$, where $c_0$ has $m_1(0)=m \in (0,\frac{1}{2})$ and also bounded zeroth and second moments, that is,
\[
m_0(0)=\int_0^\infty  c(s,0)\,ds \leq C \quad\text{and} \quad m_2(0)=\int_0^\infty s^2 c(s,0)\,ds \leq C\,.
\]
Then \eqref{eq:CF} has a mass-conserving weak solution in the measure sense.
\end{theorem}

Of course, this mass-conserving weak solution in the measure sense is unique thanks to Corollary \ref{cor:m leq 1}.
The range we get here for $0<m_1(0)<\frac{1}{2}$ is an improvement to the previous range of $0<m_1(0) < \frac{1}{4 \log 2}$ obtained in \cite{LAURENCOT2019-2}.
The proof of Theorem \ref{thm: existence of solution to C-F} is given in Subsection \ref{subsec: F Bernstein}.
Basically, under the assumptions of Theorem \ref{thm: existence of solution to C-F}, we first need to show that $F \in C^\infty((0,\infty)^2)$ in Proposition \ref{prop: F smooth}.
Then, we deduce that $(-1)^{n+1} \partial^n_x F \geq 0$ for all $n\in \N$ in Proposition \ref{prop:F changing signs derivatives}.
These highly nontrivial regularity results of $F$, together with a characterization of Bernstein functions (see Appendix~\ref{A:bernstein}), allow us to obtain Theorem \ref{thm: existence of solution to C-F}.
\medskip

We then obtain the following large time behavior result for $F$ in case $0<m<1$.
Here, we do not need  assumption \eqref{A:secondDerivativeBound2}.

\begin{theorem}\label{t:largeTimeLessThan1}
 Assume {\rm \eqref{A:firstDerivativeBound}--\eqref{A:secondDerivativeBound1}}.
  Let $0<m<1$, $F_0$ be sublinear, and $ 0 \leq F_0(x) \leq mx$.
  Let $F$ be the Lipschitz, sublinear solution to equation~\eqref{e:main}. 
  Then
  \begin{equation}
    \lim_{t\to \infty} F(x,t) = mx
  \end{equation}
  locally uniformly on $[0,\infty)$.
\end{theorem}

Heuristically, Theorem~\ref{t:largeTimeLessThan1} implies that as $t\to \infty$, all the solutions (mass-conserving or not) will turn to dusts (particles of size zero) if their initial total mass is less than $1$. To see this, we note that, if $F_\infty(x)=\lim_{t\to \infty} F(x,t)$ is a Bernstein transform, then for some measure $\mu_\infty$,
\begin{equation*}
    F_\infty (x) = \int_0^\infty ( 1 - e^{-sx} )\, \mu_\infty(ds)  = mx \,.
\end{equation*}
Differentiating in $x$, it is necessary that
\begin{equation*}
    \int_0^\infty s e^{-sx}  \, \mu_\infty(ds) = m\,,
\end{equation*}
which implies $s\mu_\infty(ds) = m \delta_0 (ds)$.

\medskip

 To avoid any confusion, we conclude the introduction by emphasizing the following points.
  \begin{itemize}
      \item  While the viscosity solution to the Hamilton-Jacobi equation \eqref{e:main} itself does not correspond to any extension of weak solutions to the C-F, if the viscosity solution $F$ is smooth (i.e., a smooth classical solution) and $(-1)^{n+1} \partial^n_x F \geq 0$ in $(0,\infty)^2$ for all $n\in \N$, it would correspond to a mass-conserving weak solution in the measure sense to the C-F.
     Therefore, regularity of the viscosity solution will imply whether one could have a mass-conserving weak solution in the measure sense to the C-F or not.
     This is, obviously, an extremely hard and central issue in the theory of viscosity solutions.
     
    \item  Here, we achieve uniqueness of mass-conserving weak solutions to the C-F for $0<m_1(0) \leq 1$. 
  We show  existence of such mass-conserving weak solutions for $0<m_1(0)<\frac{1}{2}$, and of course, the range $\frac{1}{2} \leq m_1(0) \leq 1$ is still open.
 
    \item To obtain a classical mass-conserving solution for equation~\eqref{eq:CF} in case $0<m_1(0)<\frac{1}{2}$, one needs to show that the mass-conserving weak solution in the measure sense actually admits a density, which requires more properties from the corresponding Bernstein function. 
    This has been done by Degond, Liu and Pego~\cite{DegondLiuEA17} in a different setting, but remains a hard problem here and will be addressed in future works.
  \end{itemize}

\section{Wellposedness of \eqref{e:main} in case $m \in (0, 1]$} \label{S:Wellposedness}
We first prove the existence and uniqueness of viscosity solutions to  \eqref{e:main}.
In this section, we always assume  that conditions of Theorem \ref{thm:wellposedness} are in force.

\subsection{Existence of viscosity solutions to  \eqref{e:main}}
We search for sublinear solutions to \eqref{e:main} which satisfy \eqref{ine:sublinear}, that is,
\[
0 \leq F(x,t) \leq mx \quad \text{ for all } (x,t) \in [0,\infty)^2\,.
\]
Since \eqref{e:main} is singular at $x=0$, we cut off its singularity  by introducing a sequence of function $\set{\phi_n}$ where
\begin{equation*}
\phi_n(x) = \max\left\{ \frac{1}{n},x \right\} \quad \text{ for all } x \in [0,\infty)\,.
\end{equation*}

By the classical theory of viscosity solutions, we have that for each $n\in \N$, the equation
\begin{equation} \label{e:cutoff}
  \begin{cases}
    \partial_t F + \frac{1}{2} (\partial_x F - m) ( \partial_x F - m - 1) + \frac{F}{\phi_n(x)}  -m = 0 \quad &\text{ in } (0,\infty)^2\,,\\
    F(x,0) = F_0(x) \quad &\text{ on } [0,\infty)\,,\\
    F(0,t) = 0 \quad &\text{ on } [0,\infty)\,,
  \end{cases}
\end{equation}
has a unique sublinear viscosity solution $F^n$.  
In fact, the sublinearity of $F^n$ can be  seen through the fact that
\[
F_0(x) - Ct \leq F^n(x,t) \leq F_0(x) + Ct \quad \text{ for all } (x,t) \in [0,\infty)^2\,,
\]
as $F_0(x) - Ct, F_0(x) + Ct$ are a subsolution and a supersolution to \eqref{e:cutoff}, respectively, for some $C>0$ sufficiently large.
To see this, we have
  \begin{equation*}
    C + \frac{1}{2} ( \partial_x F_0(x) - m)(\partial_x F_0(x) -m -1) + \frac{ F_0(x) + Ct}{\phi_n(x)} - m \geq 0  
  \end{equation*}
  and
  \begin{equation*}
    - C + \frac{1}{2} ( \partial_x F_0(x) - m)(\partial_x F_0(x) -m -1) + \frac{ F_0(x) - Ct}{\phi_n} - m \leq 0 
  \end{equation*}
 provided that 
 \begin{equation*}
   C \geq 2m+ \sup_{x \in (0,\infty)}  \left|( \partial_x F_0(x) - m)(\partial_x F_0(x) -m -1)\right|\,.
 \end{equation*}

\begin{lemma} \label{lem:monotone}
  For each $n \in \N$, let $F^n$ be the viscosity solution to equation~\eqref{e:cutoff}. Then, we have that
  \begin{equation} \label{ine:Fn}
    F^{n+1} \leq F^n 
  \end{equation}
  for all $n \in \N$.
\end{lemma}
\begin{proof}
    To see this, we note that $\phi_n \geq \phi_{n+1}$. Therefore
    \begin{equation*}
        \frac{F^n}{\phi_n} \leq \frac{F^n}{\phi_{n+1}} \,,
    \end{equation*}
    which implies that $F^n$ is a supersolution to equation~\eqref{e:cutoff} with $\phi_{n+1}$. Thus, \eqref{ine:Fn} follows.
\end{proof}

\begin{lemma} \label{lem:approximateLipschitz}
  For each $n \in \N$, let $F^n$ be the viscosity solution to equation~\eqref{e:cutoff}. 
  Then, $\{F^n\}$ is equi-Lipschitz, that is, there exists a constant $C>0$ so that for every $n\in \N$,
  \begin{equation} \label{ine:approximateLipschitz}
    \abs{ F^n(x_1,t_1) - F^n(x_0,t_0)} \leq C( \abs{ t_1 - t_0} + \abs{x_1 - x_0}) \,,
  \end{equation}
  for every $t_0, t_1, x_0, x_1 \in [0,\infty)$.
\end{lemma}

\begin{proof}
  We achieve global Lipschitz property in time using the solutions to the approximation problems.
  We note that equation~\eqref{e:cutoff} obeys the classical theory of viscosity solutions so
  the comparison principle holds. 

  For each $n\in \N$, we have that $\phi^- \equiv 0$ is a subsolution and $\phi^+ = mx+ \frac{1}{n}$ is a supersolution to equation~\eqref{e:cutoff}.
  To see the subsolution, we have that
  \begin{equation*}
    \frac{1}{2}m(m+1) - m = \frac{m(m-1)}{2} \leq 0 \,.
  \end{equation*}
 
  To see the supersolution, we have that
  \begin{equation*}
    \frac{ mx + \frac{1}{n}}{\phi_n(x)} - m = 
    \begin{cases}
      \frac{1}{nx} & \text{ if } x \geq \frac{1}{n}\,,  \\
      nmx + 1 - m & \text{ if } x \leq \frac{1}{n}\,,
    \end{cases}
  \end{equation*}
which is always nonnegative.
  On the other hand, as shown just before Lemma \ref{lem:monotone},
   we also have that $F_0(x) - Ct$ and $F_0(x) + Ct$ are a subsolution and a supersolution to \eqref{e:cutoff}, respectively.
  Therefore, $G^-(x,t) \defeq \max\set{ 0, F_0(x) - Ct }$ is also a subsolution, and $G^+(x,t) \defeq \min\set{ mx+\frac{1}{n}, F_0(x) + Ct}$ is also a supersolution to \eqref{e:cutoff}.
  And so, by the comparison principle,
  \begin{equation}\label{bound-F-n}
    G^-(x,t) \leq F^n(x,t) \leq G^+(x,t) \,.
  \end{equation}
  Thus, for $t >0$,
  \begin{equation*}
    \abs{ F^n(x,t) - F^n(x,0) } \leq Ct \,.
  \end{equation*}
  By the $L^\infty$-contractive property of solutions to Hamilton-Jacobi equations (which follows from the comparison principle itself), for every $t_0,  t_1 \in [0,\infty)$ with $t_1>t_0$,
  \begin{equation} \label{ine:timeLipschitz}
    \sup_x\abs{ F^n(x,t_1) - F^n(x,t_0)} \leq  \sup_x\abs{ F^n(x,t_1-t_0) - F^n(x,0)} \leq C \abs{t_1 - t_0} \,.
  \end{equation}

 This is equivalent to the fact that 
 \begin{equation} \label{ine:tderivativeBound}
   \abs{\partial_t F^n(x,t) } \leq C
 \end{equation}
 in the viscosity sense.
 Therefore, rearranging equation~\eqref{e:cutoff} and using triangle inequality, estimates \eqref{bound-F-n} and~\eqref{ine:tderivativeBound}, we have
 \begin{align*}
 \abs{(\partial_x F^n - m) (\partial_x F^n -m - 1)} = 2 \abs[\Big]{- \partial_t F^n + m - \frac{F^n}{\phi_n(x)}} \leq  C
 \end{align*}
 in the viscosity sense. Therefore, there exists a constant $C>0$ (independent of $n \in \N$) so that
 \begin{equation*}
   \abs{\partial_x F^n} \leq C
 \end{equation*}
 in the viscosity sense, which is equivalent to 
 \begin{equation} \label{ine:spaceLipschitz}
   \abs{ F^n(x_1,t) - F^n(x_0,t) } \leq C\abs{ x_1 - x_0} 
 \end{equation}
 for every $x_1, x_0 \in (0,\infty)$. Combining estimates~\eqref{ine:timeLipschitz} and~\eqref{ine:spaceLipschitz}, we get the desired inequality~\eqref{ine:approximateLipschitz}.
  \end{proof}

  \begin{lemma}\label{lem:existence}
    There exists a function $F$ so that $\set{F^n}$ converges to $F$ locally uniformly on $[0,\infty)^2$, and $F$ is sublinear, uniformly Lipschitz with the same Lipschitz constant as in Lemma~\ref{lem:approximateLipschitz}. 
    Furthermore, $F$ is a viscosity solution to equation~\eqref{e:main}.
  \end{lemma}
  \begin{proof}
    The locally uniform convergence follows from Lemmas~\ref{lem:monotone} and~\ref{lem:approximateLipschitz}.
    It is clear from the convergence and \eqref{bound-F-n} that $F$ is sublinear, and $0\leq F(x,t) \leq mx$ for all $(x,t)\in [0,\infty)^2$.
    The fact that $F$ is a viscosity solution to~\eqref{e:main} follows directly from the definition and the facts that
    $\set{F^n}$ converges to $F$ locally uniformly and $\set{\phi_n}$ converges to $x$ uniformly.
  \end{proof}

  \subsection{Uniqueness of solutions to  \eqref{e:main}}
  \begin{lemma}[Comparison Principle] \label{lem:CP}
    Let $u$ be a sublinear viscosity subsolution and $v$ be a sublinear viscosity supersolution to equation~\eqref{e:main}, respectively. Then $u\leq v$.
  \end{lemma}

  \begin{proof}
    We have that for every $n \in \N$, $u$ is a subsolution, and $v^n \defeq v + \frac{1}{n}$ is a supersolution to equation~\eqref{e:cutoff}, respectively.
    The subsolution is clear to see. 
    
    To check the supersolution property, we note that, since $m\leq 1$,
    \begin{equation*}
      \frac{v + \frac{1}{n}}{\phi_n} -m = \begin{cases}
        \frac{v}{x} + \frac{1}{nx} - m\geq \frac{v}{x} - m, &\text{ for } x \geq \frac{1}{n} \,, \\
        nv + 1 -m \geq 0 \geq \frac{v}{x} -m, &\text{ for } x < \frac{1}{n} \,.
      \end{cases}
    \end{equation*}
    Therefore, 
    \begin{align*}
     & \partial_t v^n + \frac{1}{2} (\partial_x v^n - m) ( \partial_x v^n -m - 1) + \frac{v^n}{\phi_n(x)} - m \\
     &\geq  \partial_t v + \frac{1}{2} (\partial_x v - m) ( \partial_x v -m - 1) + \frac{v}{x} - m  \geq 0
    \end{align*}
    in the viscosity sense. By the classical theory of viscosity solution applied to equation~\eqref{e:cutoff}, we imply that 
    \begin{equation*}
      u \leq v^n \,.
    \end{equation*}
    But as $v^n \to v$ locally uniformly as $n\to \infty$, we then conclude
    \begin{equation*}
      u \leq v \,,
    \end{equation*}
   as desired. 
  \end{proof}


Let us now give the proof of Theorem \ref{thm:wellposedness}.
\begin{proof}[{\bf Proof of Theorem \ref{thm:wellposedness}}]
By Lemma \ref{lem:existence}, \eqref{e:main} admits a solution $F$, which is Lipschitz on $[0,\infty)^2$, and is sublinear in $x$.
Lemma \ref{lem:CP} then yields the uniqueness of $F$.
\end{proof}

Corollary \ref{cor:m leq 1} then follows immediately.

\begin{proof}[{\bf Proof of Corollary \ref{cor:m leq 1}}]
Let $c$ be a mass-conserving solution to \eqref{eq:CF} with $m=m_1(0) \in (0,1]$.
Let $F, F_0$ be the Bernstein transforms of $c, c_0=c(\cdot,0)$, respectively.
Then, $F$ is a solution to \eqref{e:viscosity}, $F$ is sublinear in $x$, and $F \in C^\infty((0,\infty)^2) \cap C^1([0,\infty)^2)$.
In particular, $F$ is the unique sublinear viscosity solution to \eqref{e:viscosity}.
This gives the uniqueness of $c$.
\end{proof}

\section{Regularity results}

\subsection{Non-existence of $C^1$ sublinear solutions when $m>1$} \label{S:nonExistence}
We first show the impossibility of $C^1$ sublinear solutions when $m>1$.
It is important to note that the sublinear requirement is used crucially here as \eqref{e:main} admits special solutions $\psi_1(x,t) = mx$ and $\psi_2(x,t)=(m-1)x$  for all $(x,t) \in [0,\infty)^2$, which are both linear in $x$.

\begin{proof}[{\bf Proof of Theorem \ref{thm:no C1 sln}}]
  We proceed by contradiction and suppose that such a solution $F$ exists. Then,
  \begin{equation*}
    F(0,t) = 0 \quad \text{ and } \quad \partial_t F(0,t)=0\,. 
  \end{equation*} 
  Let $x \to 0^+$ in \eqref{e:main} and use the fact that 
  \[
  \partial_x F(0,t)=\lim_{x \to 0^+}\frac{F(x,t)-F(0,t)}{x}=\lim_{x \to 0^+}\frac{F(x,t)}{x}
  \]
  to yield
  \begin{equation*}
     \frac{1}{2}( \partial_x F(0,t) - m)( \partial_x F(0,t) - m -1) + \partial_x F(0,t) - m =0 \,.
  \end{equation*}
  Thus, either $\partial_x F(0,t) = m$ or $\partial_x F(0,t) = m-1$. In other words, $\partial_x F(0,t) \geq m-1>0$.
  Now, fix $\sigma \in (0, m-1)$. By sublinearity in $x$ of $F$, for a fixed $t>0$, there exists $x_t>0$ such that
  \begin{equation*}
    \varphi (t) \defeq \max_{x\in [0,\infty)} \left( F(x,t) - \sigma x \right) = F(x_t, t) - \sigma x_t > 0 \,.
  \end{equation*}
  The computations from here to the end of this proof are all justified in the viscosity sense.
  Observe that, at $x = x_t$, $\partial_x F(x_t,t) = \sigma$ and $F(x_t,t)/x_t >\sigma$. Therefore,
  \begin{equation*}
    \partial_t F(x_t,t) \leq -\frac{1}{2} (\sigma - m)(\sigma -m -1) - (\sigma - m) = -\frac{1}{2} (\sigma - m)(\sigma - m +1) \defeq -c_0 \,.
  \end{equation*}
  Furthermore,
  \begin{align*}
    \varphi'(t) &= \lim_{s \to 0^+} \frac{ \varphi(t) - \varphi(t -s)}{s} \\
    &= \lim_{s \to 0^+} \frac{ [F( x_t, t) - \sigma x_t ] - [F(x_{t-s}, t-s) - \sigma x_{t-s}]}{s} \\
    &\leq \lim_{s \to 0^+} \frac{ [ F( x_t, t) - \sigma x_t ] - [ F(x_t, t- s) - \sigma x_t ]}{ s} \\
    &= \partial_t F(x_t,t) \leq -c_0 <0\,.
  \end{align*}
  Therefore, there exists $T>0$ so that $\varphi(T) <0$, which is a contradiction.
\end{proof}

\begin{proof}[\bf{Proof of Corollary \ref{cor:m>1}}]
Assume by contradiction that there exists a mass-conserving solution $c$ to \eqref{eq:CF} with $m=m_1(0) >1$.
Let $F, F_0$ be the Bernstein transforms of $c, c_0=c(\cdot,0)$, respectively.
Then, $F$ is a solution to \eqref{e:viscosity}, $F$ is sublinear in $x$, and $F \in C^\infty((0,\infty)^2) \cap C^1([0,\infty)^2)$.
This of course contradicts Theorem \ref{thm:no C1 sln}.
The proof is complete.
\end{proof}

\subsection{The case $0<m \leq 1$} \label{S:m between 0 and 1}

In the case $0<m\leq 1$, a central topic we set out to study is when is it that classical solutions to the equation~\eqref{e:main} exist for all time.
This is not a simple task as viscosity solutions to Hamilton-Jacobi equations are Lipschitz, but might not be $C^1$ in general.

To do this, we study another regularized version of equation~\eqref{e:main} by adding a viscosity term and then study the vanishing viscosity limit. Specifically, for $\epsilon>0$, we consider
\begin{equation} \label{e:viscosity}
   \begin{cases}
     \partial_t F  + \frac{1}{2}(\partial_x F - m)(\partial_x F - m -1) + \frac{F}{x} - m &= \epsilon a(x) \partial_{xx}F \,, \\
    F(x,0) &= F_0(x) \,, \\
    F(0,t) &= 0 \,. \\
  \end{cases}
\end{equation}
In this section, we use assumptions \eqref{A:firstDerivativeBound}--\eqref{A:secondDerivativeBound2} whenever needed.

\medskip

We give ourselves some freedom of choices for the nonnegative function $a(x)$. This freedom gives us some flexibility in proving bounds.

\subsubsection{Concavity of $F$ when $0<m \leq 1$}

In this section, we always  assume {\rm \eqref{A:firstDerivativeBound}--\eqref{A:secondDerivativeBound1}}, and $F_0$ is sublinear, and $ 0 \leq F_0(x) \leq mx$.
For each $\epsilon>0$, let $F^\epsilon_1$ be the classical solution to equation~\eqref{e:viscosity} corresponding to $a \equiv 1$.
By regularity theory for parabolic equations, $F^\epsilon_1 \in C^\infty((0,\infty)^2) \cap C^{2}_1([0,\infty)\times (0,\infty))$ (for example, see Ladyženskaja, Solonnikov, Ural’ceva~\cite{LadyzenskajaSolonnikovUralceva68}, Lieberman~\cite{Lieberman96}, Krylov~\cite{Krylov96}).
Here, $C^{2}_1([0,\infty)\times (0,\infty))$ is the space of functions which  are $C^2$ in $x$ and  $C^1$ in $t$ on $[0,\infty)\times (0,\infty)$.

\begin{lemma} \label{lem:gradientbound1}
  Assume {\rm \eqref{A:firstDerivativeBound}--\eqref{A:secondDerivativeBound1}}.
  Assume further that $F_0$ is sublinear and $ 0 \leq F_0(x) \leq mx$.
  For each $\epsilon>0$, let $F^\epsilon_1$ be the classical solution to equation~\eqref{e:viscosity} corresponding to $a \equiv 1$. 
  Then, 
  \begin{equation} \label{ine:gradientbound1}
    0 \leq \partial_x F^\epsilon_1 \leq m\,.
  \end{equation}
\end{lemma}

\begin{proof}
  Firstly, as $0 \leq F^\epsilon_1(x,t) \leq mx$ for each $t\geq 0$, we imply that 
  \begin{equation}\label{eq:grad0}
  0 \leq \partial_x F^\epsilon_1(0,t) \leq m \,.
  \end{equation}
 Differentiate \eqref{e:viscosity} to get
 \[
  \cL^\epsilon[\partial_x F^\epsilon_1] + \left( \frac{\partial_x F^\epsilon_1}{x} - \frac{ F^\epsilon_1}{x^2}\right)=0\,,
 \]
   where
  \begin{equation*}
    \cL^\epsilon[\phi] \defeq \partial_t \phi + \partial_x F^\epsilon_1 \partial_x \phi - ( m + \frac{1}{2}) \partial_x \phi - \epsilon \partial_x^2 \phi 
  \end{equation*}
    is a linear parabolic operator.

By Taylor's expansion, for each $(x,t)\in (0,\infty)^2$, there exists $\alpha = \alpha(x,t) \in (0,1)$ so that 
\[
0=F^\epsilon_1(0,t) = F^\epsilon_1(x,t) - x \partial_x F^\epsilon_1(\alpha x,t) \,.
\]
Thus, 
 \[
  \cL^\epsilon[\partial_x F^\epsilon_1] +  \frac{\partial_x F^\epsilon_1(x,t)-\partial_x F^\epsilon_1(\alpha x,t)}{x}=0 \,.
 \]
 We only show here that $\partial_x F^\epsilon_1 \leq m$ by the usual maximum principle. 
 The lower bound can be done in a similar manner.
 Suppose that for some $T>0$, there exists $x_0 \geq 0$ such that
 \begin{equation*}
    \max_{ [0,\infty) \times [0,T]} \partial_x F^\epsilon_1 = \partial_x F^\epsilon_1(x_0,T) \,.
  \end{equation*}
  Thanks to \eqref{eq:grad0}, we only need to consider the case that $x_0>0$.
  At this point $\partial_x F^\epsilon_1(x_0,T) \geq \partial_x F^\epsilon_1(\alpha x_0,T)$, and so $ \cL^\epsilon[\partial_x F^\epsilon_1](x_0,T) \leq 0$.
  By repeating the proof of the maximum principle for a linear parabolic operator, we obtain the desired conclusion that $\partial_x F^\epsilon_1 \leq m$. 
\end{proof}

\begin{remark}
In the use of the maximum principle, to keep the presentation clean, it is typically the case that one assumes that maximum points of a bounded continuous function ($\partial_x F^\epsilon_1$ in the above proof) occur. To justify this point rigorously, one can consider maximum of $\partial_x F^\epsilon_1(x,t) - \delta x$ on $[0,\infty)^2$, for $\delta>0$, and let $\delta \to 0^+$.
\end{remark}

\begin{lemma}\label{lem:boundaryEst1}
  Let $F^\epsilon_1$ be the classical solution to equation~\eqref{e:viscosity} with $a \equiv 1$. 
  Then, 
  \begin{equation} \label{ine:boundaryEst1}
  \partial_x^2 F^\epsilon_1(0,t) \leq 0 \quad \text{ for all } t\geq 0 \,.
  \end{equation}
\end{lemma}

\begin{proof}
 As $F^\epsilon_1(0,t)=0$ for all $t\geq 0$, $\partial_t F^\epsilon_1(0,t) = 0$ and 
 \[
 \lim_{x\to 0^+} \frac{F^\epsilon_1(x,t)}{x} = \partial_x F^\epsilon_1(0,t)\,.
 \]
 Let $x \to 0^+$ in \eqref{e:viscosity} and use the above to get
  \begin{equation} \label{e:boundaryEq}
    \frac{1}{2} ( \partial_x F^\epsilon_1 (0,t) - m) ( \partial_x F^\epsilon_1(0,x) -  m + 1) = \epsilon \partial_x^2 F^\epsilon_1(0,t) \,,
  \end{equation}
  which, together with \eqref{ine:gradientbound1}, yields~\eqref{ine:boundaryEst1}. 
\end{proof}

We are now ready to prove that $F^\epsilon_1$ is concave in $x$.

\begin{lemma} \label{l:interiorUpperBoundApprox}
  Assume {\rm \eqref{A:firstDerivativeBound}--\eqref{A:secondDerivativeBound1}}.
  Assume further that $F_0$ is sublinear and $ 0 \leq F_0(x) \leq mx$.
  For each $\epsilon>0$, let $F^\epsilon_1$ be the classical solution to equation~\eqref{e:viscosity} corresponding to $a \equiv 1$. 
  Then, for $(x,t) \in (0,\infty)^2$,
  \begin{equation} \label{ine:interiorUpperBoundApprox}
    \partial_x^2 F^\epsilon_1 \leq 0\,.
  \end{equation}
\end{lemma}
\begin{proof}
  We proceed by the maximum principle. 
  Differentiating \eqref{e:viscosity} twice in $x$, we get 
  \begin{equation} \label{e:parabolic}
    \cL^{\epsilon} [ \partial_x^2 F_1^\epsilon ] +  (\partial_x^2 F_1^\epsilon)^2
  + \paren[\Big]{ \frac{\partial_x^2 F^\epsilon_1}{x} + \frac{2( F^\epsilon_1 - x \partial_x F^\epsilon_1)}{x^3}    } = 0 \,.
  \end{equation}
  Recall that
  \begin{equation*}
    \cL^\epsilon[\phi] = \partial_t \phi + \partial_x F^\epsilon_1 \partial_x \phi - ( m + \frac{1}{2}) \partial_x \phi - \epsilon \partial_x^2 \phi\,.
  \end{equation*}
  By Taylor's expansion, for each $(x,t)\in (0,\infty)^2$, there exists $\theta = \theta(x,t) \in (0,1)$ so that 
  \begin{equation*}
    0 = F^\epsilon_1( 0 ,t) = F^\epsilon_1 (x,t) -x \partial_x F^\epsilon_1 (x,t) + \frac{x^2}{2} \partial_x^2 F^\epsilon_1 (\theta x,t) \,.
  \end{equation*}
  This implies
  \begin{equation*}
    \frac{\partial_x^2 F^\epsilon_1}{x} + \frac{ 2(F^\epsilon_1 - x \partial_x F^\epsilon_1)}{x^3} = \frac{ \partial_x^2 F^\epsilon_1 (x,t) - \partial_x^2 F^\epsilon_1 (\theta x,t) }{x} \,,
  \end{equation*}
  which, by plugging into equation~\eqref{e:parabolic}, gives us
  \begin{equation*}
    \cL^\epsilon [ \partial_x^2 F^\epsilon_1 ] + (\partial_x^2 F^\epsilon_1)^2 + \frac{ \partial_x^2 F^\epsilon_1(x,t) - \partial_x^2 F^\epsilon_1 (\theta x, t)}{ x} = 0 \,.
  \end{equation*}
  Let us now show that $\partial_x^2 F^\epsilon_1 \leq 0$ by the usual maximum principle.
  Suppose now for some $T>0$, there exists $x_0 \geq 0$ so that  
  \begin{equation*}
    \max_{ [0,\infty) \times [0,T]} \partial_x^2 F^\epsilon_1 = \partial_x^2 F^\epsilon_1(x_0,T) \,.
  \end{equation*}
  Thanks to \eqref{ine:boundaryEst1}, we might assume further that $x_0>0$.
  By the maximum principle,
  \begin{equation*}
    \cL^\epsilon[ \partial_x^2 F^\epsilon_1] (x_0, T) \geq 0 \quad \text{ and } \quad 
    \partial_x^2 F^\epsilon_1 (x_0, T) - \partial_x^2 F^\epsilon_1 (\theta x_0, T) \geq 0 \,,
  \end{equation*}
  which yields 
  \begin{equation*}
    (\partial_x^2 F^\epsilon_1(x_0,T) )^2 \leq 0 \quad \Rightarrow  \quad \partial_x^2 F^\epsilon_1 (x_0, T) = 0 \,.
  \end{equation*}
  This implies $\partial_x^2 F^\epsilon_1 \leq 0$, as desired.
\end{proof}

Then, Lemma \ref{lem:Fconcave} is an immediate consequence of Lemmas \ref{lem:gradientbound1} and \ref{l:interiorUpperBoundApprox}.


\subsubsection{Regularity of $F$ in case $0<m<\frac{1}{2}$}

  Suppose $0<m < \frac{1}{2}$.
  Here, we always  assume {\rm \eqref{A:firstDerivativeBound}--\eqref{A:secondDerivativeBound2}}, and $F_0$ is sublinear and $ 0 \leq F_0(x) \leq mx$.
  Let $a \in C^\infty([0,\infty))$ be a nondecreasing and concave function such that
  \begin{equation}  \label{e:aDefinition}
    a(x) = \begin{cases}
      x \,, & x \in [0,1] \,, \\
      2 \,, & x \in [3,\infty) \,.
    \end{cases}
  \end{equation} 
  For each $\epsilon>0$,   let $F^\epsilon_2$ be the viscosity solution to equation~\eqref{e:viscosity} corresponding to the above $a$.
  It is worth noting that in this case, \eqref{e:viscosity} is a degenerate parabolic equation, and one needs to be careful with regularity of $F^\epsilon_2$ at $x=0$.
  Of course, $F^\epsilon_2 \in C^\infty((0,\infty)^2)$, but boundary regularity is not obvious. 
  In the following, we study further properties of $F^\epsilon_2$ by using the specific structure of the equation.

\begin{lemma}\label{lem:F2concave}
For each $\epsilon>0$, let $F^\epsilon_2$ be the viscosity solution to equation~\eqref{e:viscosity} with $a$ defined as in~\eqref{e:aDefinition}. 
Then, $F^\epsilon_2$ is concave in $x$ and
\[
0 \leq \partial_x F^\epsilon_2 \leq m \quad \text{ in } (0,\infty)^2 \,.
\]
\end{lemma}

\begin{proof}
  For each $\delta>0$, consider
  \begin{equation} \label{e:delta}
   \begin{cases}
     \partial_t F  + \frac{1}{2}(\partial_x F - m)(\partial_x F - m -1) + \frac{F}{x} - m &= (\epsilon a(x)+\delta) \partial_{xx}F \,, \\
    F(x,0) &= F_0(x) \,, \\
    F(0,t) &= 0 \,. \\
  \end{cases}
\end{equation}
 Let $F^{\epsilon,\delta}_2$ be the unique solution to the above.
 Then, $F^{\epsilon,\delta}_2 \in C^\infty((0,\infty)^2) \cap C^{2}_1([0,\infty)\times (0,\infty))$.

 By repeating the proof of Lemma \ref{lem:gradientbound1}, we obtain that $0 \leq \partial_x F^{\epsilon,\delta}_2 \leq m$.
 In a similar fashion, $\partial_x^2 F^{\epsilon,\delta}_2(0,t) \leq 0$ for all $t \geq 0$ by following the proof of \eqref{ine:boundaryEst1}.
 Finally, we use the maximum principle to conclude that $F^{\epsilon,\delta}_2$ is concave in $x$.
 Indeed, replicating the proof of Lemma \ref{l:interiorUpperBoundApprox}, we find that for some $T>0$, there exists $x_0>0$ such that
 \[
 \max_{[0,\infty) \times [0,T]} \partial_x^2 F^{\epsilon,\delta}_2 = \partial_x^2 F^{\epsilon,\delta}_2 (x_0,T) \,.
 \]
 The maximum principle then gives us that
 \[
 \left(\partial_x^2 F^{\epsilon,\delta}_2 (x_0,T) \right)^2 \leq \epsilon a''(x_0)\partial_x^2 F^{\epsilon,\delta}_2 (x_0,T) \,.
 \]
 Note that $a''(x_0) \leq 0$ as $a$ is chosen to be concave. Therefore, $\partial_x^2 F^{\epsilon,\delta}_2 (x_0,T) \leq 0$.
 Let $\delta \to 0^+$ to get the desired results.
\end{proof}

\begin{lemma}\label{lem:boundaryEst2}
  For each $\epsilon>0$, let $F^\epsilon_2$ be the viscosity solution to equation~\eqref{e:viscosity} with $a$ defined as in~\eqref{e:aDefinition}. 
  Then,  $F^\epsilon_2 \in C^1([0,\infty)^2)$ and
  \begin{equation*}
    \partial_x F^\epsilon_2 (0,t) = m\,.
  \end{equation*}
  In other words, for $t\geq 0$,
  \begin{equation}\label{lowerboundat0}
    \lim_{x\to 0^+} x \partial_x^2 F^\epsilon_2(x,t) = 0 \,.
  \end{equation}
\end{lemma}

\begin{proof}
 By Lemma \ref{lem:F2concave}, $x \mapsto \partial_x F^\epsilon_2(x,t)$ is decreasing in $(0,\infty)$ and $0 \leq  \partial_x F^\epsilon_2(x,t) \leq m$, and so, $\lim_{x \to 0^+}  \partial_x F^\epsilon_2(x,t)$ exists.
 By the L'Hopital rule,
 \[
  \partial_x F^\epsilon_2(0,t)= \lim_{x \to 0^+} \frac{F^\epsilon_2(x,t) - F^\epsilon_2(0,t)}{x}=\lim_{x \to 0^+}  \partial_x F^\epsilon_2(x,t)\,,
 \]
 which means that $x \mapsto F^\epsilon_2(x,t)$ is $C^1$ on $[0,\infty)$ for each fixed $t \geq 0$.
Besides, by the results of Daskalopoulos and Hamilton~\cite{DaskalopoulosHamilton98}, Koch~\cite{koch98}, Feehan and Pop~\cite{FeehanPop13}, we yield further that, for each $T>0$, $F^\epsilon_2 \in \mathcal{C}^{2+\beta}_s([0,\infty)\times [0,T])$, and
\[
\|F^\epsilon_2\|_{\mathcal{C}^{2+\beta}_s} \leq C \|F_0\|_{\mathcal{C}^{2+\beta}_s}
\]
for some constant $C=C(\epsilon,T)>0$.
Here,
\[
\|f\|_{\mathcal{C}^{2+\beta}_s} \defeq \|f\|_{\mathcal{C}^{\beta}_s}+\|\partial_x f\|_{\mathcal{C}^{\beta}_s}+\|\partial_t f\|_{\mathcal{C}^{\beta}_s}+\|x \partial^2_x f\|_{\mathcal{C}^{\beta}_s},
\]
and
\[
\|f\|_{\mathcal{C}^{\beta}_s} \defeq \|f\|_{L^\infty([0,\infty)\times [0,T])} + \sup_{\substack{(x_1,t_1) \neq (x_2,t_2)\\(x_1,t_1),(x_2,t_2) \in [0,\infty)\times [0,T]}} \frac{|f(x_1,t_1)-f(x_2,t_2)|}{s((x_1,t_1),(x_2,t_2))^\beta}.
\]
The distance $s$ is defined as: For $(x_1,t_1),(x_2,t_2) \in [0,\infty)^2$,
\[
s((x_1,t_1),(x_2,t_2)) \defeq \frac{|x_1-x_2|}{\sqrt{x_1}+\sqrt{x_2}} + \sqrt{|t_1-t_2|}.
\]

Let us show now that in fact $\partial_x F^\epsilon_2 (0,t) = m$ for all $t\geq 0$.
 For any $0<b_1 <b_2$, denote by
 \[
 G(x) \defeq \int_{b_1}^{b_2} F^\epsilon_2(x,t)\,dt\,.
 \]
Integrate \eqref{e:viscosity} with respect to $t \in [b_1, b_2]$ and let $x \to 0^+$ to yield
  \begin{equation*}
    \lim_{x\to 0^+} \epsilon x\partial_x^2 G(x) = \frac{1}{2} \int_{b_1}^{b_2}( \partial_x F^\epsilon_2 (0,t) - m) ( \partial_x F^\epsilon_2 (0,t) - m + 1)\,dt \leq 0 \,.
  \end{equation*}
 Suppose by contradiction that the right hand side above is negative, which is denoted by $-C<0$. 
 Then, 
  \begin{equation*}
    \lim_{x \to 0^+} x \partial_x^2 G (x) = - \frac{C}{\epsilon} < 0 \,.
  \end{equation*}
  Thus, by the L'Hopital rule,
  \begin{align*}
    -\frac{C}{\epsilon} &= \lim_{x \to 0^+} x \partial_x^2 G (x) = \lim_{x \to 0^+} \frac{ \partial_x^2 G (x)}{1/x} = \lim_{x \to 0} \frac{ \partial_x G (x)}{ \log x} \,.
  \end{align*}
  However, note that
  \[
  \abs{ \partial_x G (x) } = \left|\int_{b_1}^{b_2} \partial_x F^\epsilon_2(x,t)\,dt \right| \leq m(b_2-b_1)=C \,,
  \]
  which implies that 
  \begin{equation*}
    \lim_{x \to 0} \frac{ \partial_x G (x)}{\log x} = 0 \,,
  \end{equation*}
  which is a contradiction.
  Thus, we always have  $\lim_{\epsilon \to 0^+} \epsilon x \partial_x^2 G(x)=0$ for any $0<b_1 <b_2$ and, therefore, $\partial_x F^\epsilon_2(0,t)=m$ for all $t\geq 0$.
  This gives us \eqref{lowerboundat0} and also that
  \[
  \lim_{x \to 0^+} \partial_t F^\epsilon_2(x,t) = 0 = \partial_t F^\epsilon_2(0,t)\,.
  \]
  The proof is complete.
\end{proof}

\begin{lemma} \label{l:interiorLowerBoundApprox}
  For each $\epsilon>0$, let $F^\epsilon_2$ be the viscosity solution to equation~\eqref{e:viscosity} with $a$ defined as in~\eqref{e:aDefinition}. 
  Then,  for $\epsilon>0$ sufficiently small, 
  \begin{equation} \label{ine:interiorLowerBoundApprox}
    x\partial_x^2 F^\epsilon_2 \geq -1  \quad \text{ in } (0,\infty)^2\,.
  \end{equation}
\end{lemma}

\begin{proof}
We break the proof into a few steps as following.

\smallskip

  \emph{ Step 1.} Again, differentiating~\eqref{e:viscosity} twice in $x$, we get
  \begin{equation} \label{e:twiceDiff}
    \begin{aligned}
      \paren[\Big]{ \partial_t\partial_x^2 F^\epsilon_2 + \left[\partial_x F^\epsilon_2 - (m+\frac{1}{2})\right] \partial_x^3 F^\epsilon_2  } + (\partial_x^2 F^\epsilon_2)^2 + \frac{ \partial_x^2 F^\epsilon_2}{x}
    - \frac{ 2 \partial_x F^\epsilon_2}{x^2} + \frac{ 2 F^\epsilon_2}{x^3} \\
      = \epsilon \paren[\Big]{ a'' \partial_x^2 F^\epsilon_2 + 2 a' \partial_x^3 F^\epsilon_2 + a \partial_x^4 F^\epsilon_2  } \,.
    \end{aligned}
  \end{equation}

  Let
  \begin{equation*}
    G^\epsilon \defeq x \partial_x^2 F^\epsilon_2 \,.
  \end{equation*}
  By concavity of $F^\epsilon_2$ in $x$ (Lemma \ref{lem:F2concave}) and the proof of Lemma \ref{lem:boundaryEst2},
  $G^\epsilon \in \mathcal{C}^{\beta}_s([0,\infty)\times [0,T])$ for each $T>0$, $G^\epsilon\leq 0$, and $G^\epsilon(0,t)=0$ for all $t \geq 0$.
  Besides,  by the condition~\eqref{A:secondDerivativeBound2}, we have that
 \[
  -\frac{1}{4} \leq -\frac{m}{e} \leq x F_0''(x) = G^\epsilon(x,0)  \leq 0 \quad \text{ for all }  x \geq 0 \,.
 \]
  For $t\geq 0$, denote by
  \[
  \alpha(t)  \defeq \inf_{[0,\infty) \times [0,t]} G^\epsilon.
  \]
  Surely, $\alpha:[0,\infty) \to (-\infty,0]$ is decreasing and bounded, and $\alpha(0) \in [-\frac{1}{4},0]$.
  We now show that $\alpha$ is continuous.
  Fix $T>0$.
  For $s,t \in [0,T]$, we use the property $G^\epsilon \in \mathcal{C}^{\beta}_s([0,\infty)\times [0,T])$ to see that, for each $x >0$,
  \[
  |G^\epsilon(x,s) - G^\epsilon(x,t)| \leq C |s-t|^{\beta/2},
  \]
  for some $C=C(\epsilon,T)>0$.
  Therefore, for $s,t \in [0,T]$, 
    \begin{equation}\label{eq: alpha cont}
   |\alpha(s) - \alpha(t)| \leq C |s-t|^{\beta/2}.
  \end{equation}
 Thus, $\alpha$ is locally H\"older continuous, and hence, is continuous on $[0,\infty)$.
 It is of our goal now to show that $\alpha(t) \geq -1$ for all $t \in [0,\infty)$.

  \smallskip
  
    \emph{ Step 2.} 
    Fix $T>0$ such that $\alpha(T) < \alpha(0)$.
  Suppose that there exists $(x_0,t_0) \in (0,\infty) \times (0,T]$ such that  
  \begin{equation*}
    \min_{[0,\infty) \times [0,T]} G^\epsilon(x,t) = G^\epsilon(x_0, t_0) = \alpha(T) < 0 
  \end{equation*}
  (see Remark~\ref{remark:maximumprinciple}).
  We then have that, at $(x_0, t_0)$,
  \[
  0\geq \partial_t G^\epsilon = x_0 \partial_t \partial_x^2 F^\epsilon_2\,,
  \]
  and
  \begin{equation} \label{first derivative 0}
    0 = \partial_x G^\epsilon = x_0 \partial_x^3 F^\epsilon_2 + \partial_x^2 F^\epsilon_2  \iff \partial_x^2 F^\epsilon_2 = - x_0 \partial_x^3 F^\epsilon_2\, ,
  \end{equation}
  and, therefore, 
  \begin{equation} \label{ine:critical}
    0 \leq \partial_x^2 G^\epsilon \iff x_0^2 \partial_x^4 F^\epsilon_2 \geq - 2x_0 \partial_x^3 F^\epsilon_2 = 2 \partial_x^2 F^\epsilon_2 \,.
  \end{equation}

  Multiplying equation~\eqref{e:twiceDiff} by $x_0^2$ and use estimate~\eqref{ine:critical} to evaluate at $(x_0,t_0)$, we obtain
  \begin{align*}
  &\alpha(T)^2 + \alpha(T) \left( m + \frac{3}{2} - \partial_x F^\epsilon_2\right) + \frac{2 ( F^\epsilon_2 - x_0 \partial_x F^\epsilon_2)}{x_0}\\
    \geq \, & \epsilon \alpha(T) \paren[\Big]{ \frac{2 a(x_0)}{x_0} - 2a'(x_0) + a''(x_0) x_0   } \geq 2 \epsilon \alpha(T) \,.
  \end{align*}
  The last inequality follows since $\alpha(T) \leq 0$ and,   by the way we choose $a$,
  \[
   \frac{2 a(x_0)}{x_0} - 2a'(x_0) + a''(x_0) x_0 \leq    \frac{2 a(x_0)}{x_0}  \leq  2  \,.
     \]
  Therefore, rearranging terms,  we have 
  \begin{equation*}
    \alpha(T)^2 + A \alpha(T)
    + B \geq 0 \,,
  \end{equation*}
  where
 \begin{equation*}
   A =  m + \frac{3}{2} - 2 \epsilon - \partial_x F^\epsilon_2(x_0,t_0) \,,  
 \end{equation*}
 and
 \begin{equation*}
  B = \frac{ 2( F^\epsilon_2(x_0,t_0) - x_0 \partial_x F^\epsilon_2(x_0,t_0))}{x_0} \,.
 \end{equation*}

 We have that, since $0 \leq \partial_x F^\epsilon_2 \leq m$ and $F^\epsilon_2$ is concave in $x$, for $\kappa= m-\partial_x F^\epsilon_2(x_0,t_0)$,
 \[
 0 \leq \kappa \leq m \quad \text{ and } \quad 0 \leq B \leq 2\kappa \,.
 \]
 Therefore,
 \begin{equation*}
   \frac{3}{2} +\kappa - 2\epsilon = A \leq m + \frac{3}{2} - 2\epsilon <2 \,,
 \end{equation*}
 and 
 \begin{equation*}
   0 \leq B \leq 2\kappa \leq 2m \,.
 \end{equation*}
 As $0<m<\frac{1}{2}$, obviously $0<\kappa \leq m<\frac{1}{2}$. For $\epsilon>0$ sufficiently small, 
 \begin{align*}
 A^2-4B &\geq  \left(\frac{3}{2}+\kappa - 2\epsilon\right)^2 - 8 \kappa \geq \frac{9}{4}+\kappa^2 - 5\kappa-  8\epsilon\\
 &=\left(\frac{1}{2}-\kappa\right)\left(\frac{9}{2}-\kappa\right) - 8 \epsilon\\
 & \geq \left(\frac{1}{2}-\kappa\right)^2 + 4\left(\frac{1}{2}-m\right)- 8 \epsilon >\left(\frac{1}{2}-\kappa\right)^2>0 \,.
 \end{align*}
 From the quadratic formula and the above estimates, we find that either 
 \begin{equation*}
   \alpha(T) \leq \frac{ - A - \sqrt{ A^2-4B }}{2}\,,
 \end{equation*}
 or
 \begin{equation*}
   \alpha(T) \geq \frac{ - A + \sqrt{ A^2-4B }}{2}  \,.
 \end{equation*}
It is worth noting that, for $\epsilon>0$ sufficiently small, 
\[
\frac{ - A - \sqrt{ A^2-4B }}{2} \leq \frac{ - A - \left(\frac{1}{2}-\kappa\right)}{2} = -1 +\epsilon \leq -\frac{3}{4} - \frac{m}{2}\,,
\]
and
\[
\frac{ - A + \sqrt{ A^2-4B }}{2} \geq \frac{ - A + \left(\frac{1}{2}-\kappa\right)}{2} = -\frac{1}{2} - \kappa +\epsilon \geq-\frac{1}{2}-m\,.
\]
We then deduce that, for each $T \geq 0$, either
\begin{equation}\label{alpha-case 1}
   \alpha(T) \leq-\frac{3}{4} - \frac{m}{2}\,,
 \end{equation}
 or
 \begin{equation}\label{alpha-case 2}
   \alpha(T) \geq -\frac{1}{2}-m\,.
 \end{equation}
Surely, $-\frac{3}{4} - \frac{m}{2} < -\frac{1}{2}-m$ and there is  a gap of size $\frac{1-2m}{4}$ between these two numbers.

\medskip

 \emph{ Step 3.} 
 We show that, for small enough $\epsilon>0$, only \eqref{alpha-case 2} holds for all $T \geq 0$.
 Assume by contradiction that this  is not the case, then there exists $T>0$ such that \eqref{alpha-case 1} holds, that is,
 \begin{equation*}
   \alpha(T) \leq-\frac{3}{4} - \frac{m}{2}< -\frac{1}{2}-m < \alpha(0)\,,
 \end{equation*}
 By the continuity of $\alpha$, there exists $T^\epsilon \in (0,T)$ so that
 \begin{equation*}
-\frac{3}{4} - \frac{m}{2}<\alpha(T^\epsilon) = \min_{[0,T^\epsilon]} \alpha <  -\frac{1}{2}-m \,,
 \end{equation*}
 which is a contradiction with the conclusion of Step 2 above.

\smallskip

 Thus, for small enough $\epsilon>0$, 
 \begin{equation*}
   x \partial_x^2 F^\epsilon_2(x,t) \geq  -\frac{1}{2} - m > -1
 \end{equation*}
 for every $(x,t) \in (0,\infty)^2$, as desired.
\end{proof}

\begin{remark} \label{remark:maximumprinciple}
In the use of the maximum principle in the above proof, to keep the presentation clean, we assume that minimum points of $G^\epsilon$, which is continuous and bounded, exist on $[0,\infty) \times [0,T]$ for $T>0$.
To justify this point rigorously, one can consider minimum of $G^\epsilon(x,t) + \delta x$, for $\delta>0$, and let $\delta \to 0^+$.
Let us supply the details here.

\smallskip

Pick $T>0$ such that
\[
\alpha(T) = \min_{[0,T]}\alpha < \alpha(0).
\]
For each $k \in \N$ sufficiently large, we choose $\delta_k \in (0,\frac{1}{k})$ sufficiently small such that
\[
\alpha(T) \leq \min_{[0,\infty) \times [0,T]} \left( G^\epsilon(x,t) +\delta_k x \right)  = G^\epsilon(x_k,t_k) +\delta_k x_k \leq \alpha(T) +\frac{1}{k}<\alpha(0),
\]
for some $(x_k,t_k) \in (0,\infty)\times (0,T]$.
In particular, $\delta_k x_k \leq \frac{1}{k}$.
Let 
\[
\alpha_k  = G^\epsilon(x_k,t_k)  \in \left(\alpha(T),\alpha(T)+\frac{1}{k}\right).
\]
We use the maximum principle at $(x_k,t_k)$ and perform careful computations to deduce that
  \begin{align*}
  &\alpha_k^2 + \alpha_k \left( m + \frac{3}{2} - \partial_x F^\epsilon_2\right) + \frac{2 ( F^\epsilon_2 - x_k \partial_x F^\epsilon_2)}{x_k} + \delta_k x_k \left( m + \frac{1}{2} +2\epsilon a'(x_k)- \partial_x F^\epsilon_2\right) \\
&   \geq \epsilon \alpha_k \paren[\Big]{ \frac{2 a(x_k)}{x_k} - 2a'(x_k) + a''(x_k) x_k   } \geq 2 \epsilon \alpha_k \,.
  \end{align*}
  Let $k \to \infty$ and argue in a similar way as in Step 2 of the above proof to yield that either
\[
\alpha(T) \leq -\frac{3}{4} - \frac{m}{2},
\]
or
\[
\alpha(T) \geq-\frac{1}{2}-m,
\]
from which the proof follows.
As this is of course tedious and distracting, we intentionally avoid putting it in the above already technical proof.
\end{remark}

\smallskip

We are now ready to prove one of our main regularity results that $F \in C^{1,1}((0,\infty)^2) \cap C^1([0,\infty)\times(0,\infty))$ when $0<m<\frac{1}{2}$.

\begin{proof}[{\bf Proof of Theorem \ref{thm:regularity-m-less-than-quarter}}]
   From Lemma~\ref{lem:F2concave}, Lemma~\ref{lem:boundaryEst2} and Lemma~\ref{l:interiorLowerBoundApprox}, we have that
   $\abs{\partial_x F^\epsilon_2} \leq m$, $\abs{x\partial_x^2 F^\epsilon_2} \leq 1$ and $\abs{\partial_t F^\epsilon_2} \leq C$.
   Thus, by the Arzel\`a-Ascoli theorem, there exists  $F$ in $C([0,\infty)^2)$ and a subsequence $\set{ \epsilon_i} \to 0$ so that, locally uniformly
   \begin{gather*}
       \lim_{i\to \infty} F^{\epsilon_i}_2  = F  \,.
   \end{gather*}
   By stability of viscosity solutions, $F$ solves equation~\eqref{e:main}.
   
   Now, fix $x_0 >0$. For $x > x_0$, by Lemmas~\ref{lem:F2concave} and~\ref{l:interiorLowerBoundApprox},
   \begin{equation*}
       -\frac{1}{x_0} \leq \partial_x^2 F^\epsilon_2 (x) \leq 0 \,.
   \end{equation*}
   Letting $x_1, x_2 > x_0$, we have
   \begin{equation} \label{ine:differenceQuotient}
       \abs[\Big]{ \frac{\partial_x F^\epsilon_2( x_1,t ) - \partial_x F^\epsilon_2(x_2, t)}{ x_1 - x_2}    } 
       = \abs[\Big]{ \frac{ \int_{x_2}^{x_1} \partial_x^2 F^\epsilon_2 (x,t) \, dx}{x_1 - x_2 }   } \leq \frac{1}{x_0} \,.
   \end{equation}
   Thus, there exist constants $C>0$ and $z_0 >0$, such that for $x > x_0$ and $0 <z < z_0$, we can uniformly bound the double difference quotient
   \begin{equation*}
       \abs[\Big]{ \frac{ F^\epsilon_2 ( x+ 2z, t) - 2 F^\epsilon_2 ( x+ z, t) + F^\epsilon_2(x, t)}{z^2}   } \leq \frac{ C}{x_0} \,.
   \end{equation*}
   Letting $\epsilon$ to $0$, we get
   \begin{equation*}
        \abs[\Big]{ \frac{ F ( x+ 2z, t) - 2 F ( x+ z, t) + F(x, t)}{z^2}   } \leq \frac{ C}{x_0} \,.
   \end{equation*}
   This implies $F$ is $C^{1,1}$ in $x$ on $[x_0,\infty) \times (0,\infty)$ for all $x_0 >0$, which yields that $F$ is locally $C^{1,1}$ in $x$ in $(0,\infty)^2$. 
   It is clear then that $F$ is concave and $F$ inherits estimate \eqref{ine:interiorLowerBoundApprox} from $F_2^\epsilon$, that is,
   \[
   -1 \leq x \partial_x^2 F(x,t) \leq 0 \quad \text{ for all } (x,t) \in (0,\infty)^2.
   \]
  
  On the other hand, differentiating equation \eqref{e:main} in $x$, we have
  \begin{equation*}
    \partial_{t} U +  \partial_{x} U \left( U - m -\frac{1}{2}\right) + \frac{U}{x} - \frac{F}{x^2}  = 0 
  \end{equation*}
  in the viscosity sense, where $U= \partial_x F$. 

  Now, letting $x > x_0$,
   by the obtained estimates on $F$, 
  \begin{equation*}
    0\leq U(x,t) \leq m, \quad   0 \leq \frac{F(x,t)}{x} \leq m \quad 
    \text{ and } \quad -\frac{1}{x_0}\leq \partial_x U(x,t) \leq 0  \,. 
  \end{equation*}
  Therefore,  
  there exists $C=C(x_0)$ such that for $x > x_0$,
  \begin{equation*}
    |\partial_t U(x,t) | = |\partial^2_{tx} F(x,t)| \leq C 
  \end{equation*}
  in the viscosity sense. 
  In a similar way, differentiate the equation with respect to $t$ to deduce that for $x > x_0$ and $t>0$, 
  there exists $C=C(x_0)$ such that
   \begin{equation*}
     |\partial^2_{t} F (x,t)| \leq C \,.
  \end{equation*}
   Therefore, $F \in C^{1,1}((0,\infty)^2)$, and $F$ is concave in $x$. A similar argument (but easier) as that in the proof of Lemma~\ref{lem:boundaryEst2} shows that 
   $F \in C^1([0,\infty)\times(0,\infty))$.
\end{proof}

\subsubsection{Existence of solutions to equation~\eqref{eq:CF} for $0<m_1(0)<\frac{1}{2}$} \label{subsec: F Bernstein}

We now prove the existence of mass-conserving weak solutions in the measure sense to equation~\eqref{eq:CF} when $0<m=m_1(0) <\frac{1}{2}$. Therefore, in this subsection, we will always assume $F_0$ is the Bernstein transform of $c_0=c(\cdot,0)$, where $c_0$ has $m_1(0)=m \in (0,\frac{1}{2})$ and also bounded second moment, that is,
\[
m_2(0)=\int_0^\infty s^2 c(s,0)\,ds \leq C\,.
\]

Our goal is to show, via a combination of the maximum principle and localizations around the characteristics (see Evans~\cite[Chapter 3]{Evans10}), that $F$ is a Bernstein function (see Appendix~\ref{A:bernstein}) and, therefore, has a representation as a Bernstein transform of a measure.

By Theorem \ref{thm:regularity-m-less-than-quarter}, we already have that $F \in C^{1,1}((0,\infty)^2) \cap C^1([0,\infty)\times(0,\infty))$. 
Let us now use this result to yield further that $F \in C^\infty((0,\infty)^2) \cap C^1([0,\infty)^2)$.
\begin{proposition} \label{prop: F smooth}
  Assume all the assumptions in Theorem \ref{thm: existence of solution to C-F}.
  Then $F \in C^\infty((0,\infty)^2) \cap C^1([0,\infty)^2)$.
\end{proposition}

\begin{proof}
  We proceed by using characteristics and earlier results.
  Denote by $X(x,t)$ the characteristic starting from $x$, that is, $X(x,0)=x$.
  Set $P(x,t)=\partial_x F(X(x,t),t))$, and $Z(t)=F(X(x,t),t)$ for all $t \geq 0$.
  When there is no confusion, we just write $X(t),P(t),Z(t)$ instead of $X(x,t), P(x,t), Z(x,t)$, respectively.
  Then, $X(0)=x$, $P(0)=\partial_x F_0(x)$, $Z(0)=F_0(x)$.
  We have the following Hamiltonian system
  \[
  \begin{cases}
  \dot X = \partial_p H(P(t),Z(t),X(t)) = P(t) - \left(m+\frac{1}{2}\right)\,,\\
  \dot P = -\partial_x H - (\partial_z H)P = \frac{Z(t)}{X(t)^2}- \frac{P(t)}{X(t)}\,,\\
  \dot Z = P\cdot \partial_pH - H = \frac{P(t)^2}{2} - \frac{Z(t)}{X(t)} + \frac{m(1-m)}{2}\,.
  \end{cases}
  \]
  Note first that $F \in C^{1,1}((0,\infty)^2) \cap C^1([0,\infty)\times(0,\infty))$, and also $0 \leq \partial_x F \leq m$ thanks to Theorem \ref{thm:regularity-m-less-than-quarter}.
    Therefore,
  \begin{equation} \label{eq:X dot}
        -1 \leq - \left(m+\frac{1}{2}\right) \leq \dot X \leq -\frac{1}{2}\,.
  \end{equation}
  Besides, the concavity of $F$ in $x$ yields further that
  \[
   \dot P =  \frac{Z(t)}{X(t)^2}- \frac{P(t)}{X(t)} =\frac{1}{X(t)} \left( \frac{F(X(t),t)}{X(t)} - \partial_x F(X(t),t) \right) \geq 0\,.
  \]
  
  Let us now show that $\{X(x,\cdot)\}_{x \in (0,\infty)}$ are well-ordered in $(0,\infty)^2$, and none of these two intersect. 
  Assume otherwise that $X(x,t)=X(y,t) > 0$ for some $x \neq y$ and $t>0$.
  As $F \in C^{1,1}((0,\infty)^2) \cap C^1([0,\infty)\times(0,\infty))$, $\partial_x F(X(x,t),t)$ is uniquely defined, and therefore,
  \[
  P(x,t)=P(y,t)= \partial_x F(X(x,t),t) \quad \text{ and } \quad Z(x,t)=Z(y,t) = F(X(x,t),t)\,.
  \]
  
  \begin{figure}[ht]
    \includegraphics[width=8cm]{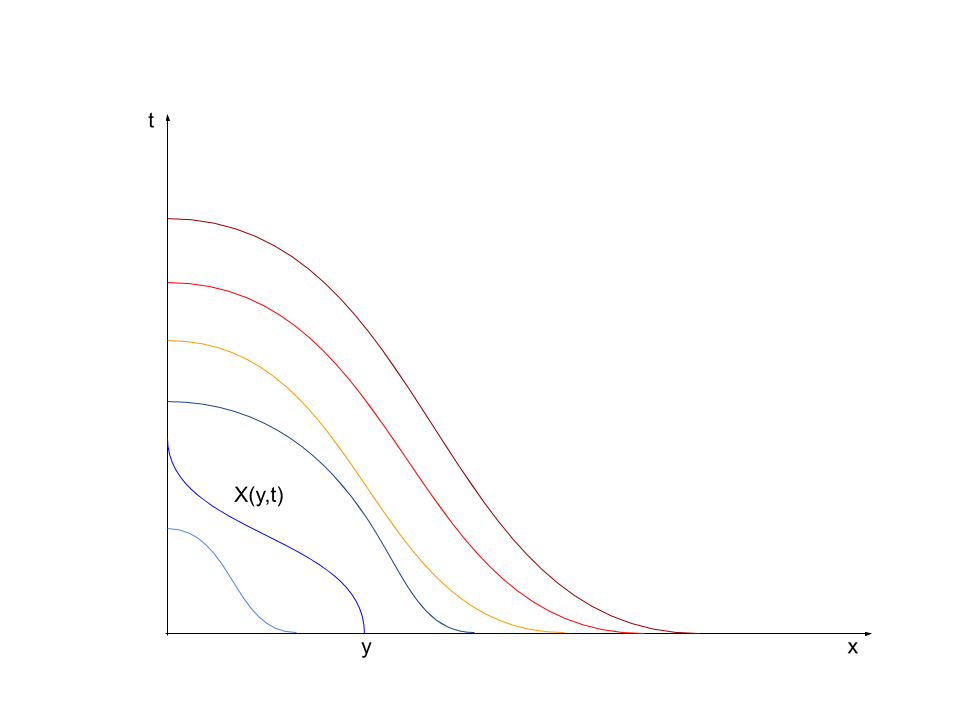}
    \centering
  \caption{Characteristics}
    \end{figure}
    
  Hence, $(X,P,Z)(x,t)=(X,P,Z)(y,t)$, and this contradicts the uniqueness of solutions to the Hamiltonian system on $[0,t]$ as we reverse the time.
  
  Next, for each $t>0$, let $l(t)>0$ be such that $X(l(t),t)=0$. 
  This is possible because of \eqref{eq:X dot}.
  As $F_0$ is smooth, $X,P,Z$ are smooth in $x$.
  Thanks to our Hamiltonian system and the well-ordered of $\{X(x,\cdot)\}_{x \in (0,\infty)}$, the map $x \mapsto X(x,t)$ is a smooth bijection from $(l(t),\infty)$ to $(0,\infty)$.
  Let $X^{-1}(\cdot,t):(0,\infty) \to (l(t),\infty)$ be the inverse of $X(\cdot,t)$.
  
  \medskip 
  
  Let us show further that $X(\cdot,t)$ is a smooth diffeomorphism.
  It is enough to show that $X(\cdot,t):(l(t)+n^{-1},n) \to (X(l(t)+n^{-1},t), X(n,t))$ is a smooth diffeomorphism for each $n\in \N$ sufficiently large.
  Let
  \[
  O=\left\{(X(x,s),s)\,:\,x \in (l(t)+n^{-1},n), s \in [0,t]\right\}.
  \]
  Thanks to Theorem \ref{thm:regularity-m-less-than-quarter}, there exists $C>0$ such that
  \[
  -C \leq \partial^2_x F(x,s) \leq 0 \qquad \text{ in } O
  \]
  in the viscosity sense.
  We differentiate the first equation in the Hamiltonian system with respect to $x$ and use the fact that $P(x,s)=\partial_x F(X(x,s),s)$ to yield that
  \[
  \partial_x \dot X(x,s) = \partial_x P(x,s) = \partial^2_x F(X(x,s),s)\cdot  \partial_x X(x,s) \geq -C \partial_x X(x,s).
  \]
  Thus, $\partial_x X(x,s)$ satisfies a differential inequality, and in particular, 
  \[
  s\mapsto e^{Cs} \partial_x X(x,s) \qquad \text{ is nondecreasing on } [0,t].
  \]
  It is then clear  that $\partial_x X(x,s) > 0$ for all $x\in (l(t)+n^{-1},n), s \in [0,t]$.
  By the inverse function theorem, $X^{-1}(\cdot,t)$ is then smooth, and
  \[
  F(x,t) = Z(X^{-1}(x,t),t)
  \]
 is smooth as $Z$ is also smooth.
 
 \smallskip
 
 Let us finally use the property $\dot P \geq 0$ to yield  that $F \in C^1([0,\infty)^2)$.
We only need to show that $\partial_x F$ is continuous at $(0,0)$.
For each $\epsilon>0$, we are able to find $r>0$ such that $F_0'(x) \in [m-\epsilon,m]$ for all $x\in [0,r]$.
Let 
\[
V_r=\{(y,s)\in [0,\infty)^2\,:\, y= X(x,s) \text{ for some  $x \in [0,r]$ and $s \geq 0$}\}.
\]
Then, as $\dot P \geq 0$, we see that $\partial_x F(y,s) \in [m-\epsilon,m]$ for all $(y,s) \in V_r$.
The proof is complete.
\end{proof}

It is worth noting that in this problem, for the characteristics, only the condition for $t=0$ is in use. The boundary condition for $x=0$, though still satisfied, is not being used (ineffective).

Now that we have $F \in C^\infty((0,\infty)^2) \cap C^1([0,\infty)^2)$, we continue to prove the last requirement to have that $F$ is a Bernstein function.

\begin{proposition} \label{prop:F changing signs derivatives}
  Assume all the assumptions in Theorem \ref{thm: existence of solution to C-F}.
  Then, 
  \[
  (-1)^{n+1} \partial^n_x F \geq 0 \quad\text{ in $(0,\infty)^2$ for all $n\in \N$\,.}
  \]
\end{proposition}

Of course, we verified the above claim already when $n=1$.
A main difficulty to achieve this result is that $\partial^n_x F$ might be singular at $x=0$, and thus, we do not have much knowledge on the boundary behavior there.
This is also clear in view of the method of characteristics as described above.
Here is a way to fix this issue, which is motivated by Lemma \ref{lem:boundaryEst2}.

\begin{lemma}\label{lem: finer on second derivative}
  We have that, for all $t \geq 0$,
  \[
  \lim_{x \to 0^+} x \partial^2_x F(x,t)=0\,.
  \]
\end{lemma}

\begin{proof}
  Let $Q=\partial^2_x F$.
  Differentiate \eqref{e:main} with respect to $x$ twice, we get
  \begin{equation}\label{eq: Fxx}
  \partial_t Q  -\left(m+\frac{1}{2}-\partial_x F\right) \partial_x Q = - Q^2 - \frac{Q}{x} + 2\frac{x \partial_x F - F}{x^3}\,.
  \end{equation}
  A very important point here is that \eqref{eq: Fxx} has the same characteristics $X(x,t)$ as in Proposition \ref{prop: F smooth}.
 Recall that
 \[
 \dot X = -\left(m+\frac{1}{2}\right) + \partial_x F(X(t),t)\,,
 \]
  and \eqref{eq:X dot} holds.
  Let $R(t) = Q(X(t),t)$, then
  \[
  \dot R = - R^2 - \frac{R}{X} + 2\frac{X P - Z}{X^3}\,. 
  \]
  Since $-1\leq x \partial^2_x F \leq 0$, we infer that $R \leq 0$, $1+RX\geq 0$, and
  \begin{equation}\label{eq: important ineq R}
  \dot R = - R^2 - \frac{R}{X} + 2\frac{X P - Z}{X^3} \geq 2\frac{X P - Z}{X^3}
  =\frac{2}{X^2}\left(P - \frac{Z}{X}\right)\,.
  \end{equation}
  This differential inequality about $R$ will be used to give us the desired result.
  Note that $F \in C^1([0,\infty)^2)$, and for each $t\geq 0$,
  \[
  \lim_{x\to 0^+} \left(\partial_x F(x,t) - \frac{F(x,t)}{x}\right)=0\,.
  \]
  So, for fixed $T>0$, there exists a modulus of continuity $\omega:(0,\infty) \to [0,\infty)$ with $\lim_{r\to 0^+} \omega(r)=0$ such that for all $r>0$,
  \[
  \left| \partial_x F(x,t) - \frac{F(x,t)}{x}\right| \leq \omega(r) \quad \text{ for all } (x,t) \in (0,r]\times [0,T]\,.
  \]
  Fix $r>0$ and on each given characteristic $X(x,\cdot)$, which reaches $0$ in finite time, take $s_0 \geq 0$ such that $0<X(x,s_0) \leq r$. 
  For $s \geq s_0$, we use this in \eqref{eq: important ineq R} to get that
  \[
  \dot R(s) \geq -\frac{2 \omega(r)}{X(s)^2}\,.
  \]
  Integrate this and use \eqref{eq:X dot} to yield, for $t \geq s_0$,
  \begin{align*} 
  R(t) &\geq R(s_0) - 2 \omega(r)\int_{s_0}^t \frac{1}{X(s)^2}\,ds \geq R(s_0) - 2 \omega(r)\int_{s_0}^t \frac{-2 \dot X(s)}{X(s)^2}\,ds\\
  &= R(s_0) - 4\omega(r) \left(\frac{1}{X(t)} - \frac{1}{X(s_0)} \right)\,.
  \end{align*} 
  Thus,
  \[
  X(t) R(t) \geq X(t) R(s_0) - 4 \omega(r)\,.
  \]
  Besides, $X(t)R(t) \leq 0$ thanks to Theorem \ref{thm:regularity-m-less-than-quarter}.
  Combine the two inequalities, we get, for $X(t) \leq r$ and $t \in [0,T]$,
  \begin{equation}\label{another modulus}
  |X(t)R(t)| \leq CX(t) + 4\omega(r),
  \end{equation}
  where $C=\max_{x\in [0,r]} |F_0''(x)| + \max_{t\in [0,T]} |\partial_x^2 F(r,t)|$.
  Let $X(t) \to 0^+$ and $r \to 0^+$ in this order in the above to get the conclusion.
\end{proof}

\begin{lemma}\label{lem:bound of higher derivatives}
  Fix $n\in \N$ with $n \geq 2$, and $R>0$.
  Then, there exists a constant $C=C(n,R)>0$ such that
  \begin{equation}\label{bhd}
  \|x^{n-1} \partial^n_x F(x,t)\|_{L^\infty((0,R)^2)} \leq C\,.
  \end{equation}
\end{lemma}

\begin{proof}
  The proof is rather tedious with a lot of terms appearing in the differentiations.
  We prove by induction with respect to $j=n$ in \eqref{bhd}.
  The base case $j=2$ was already done by Theorem \ref{thm:regularity-m-less-than-quarter}.
  Assume that \eqref{bhd} holds true for $j=n-1 \geq 2$, and we now show that it is also true for $j=n$. 
  
  \emph{Step 1.}
  Differentiate \eqref{e:main} with respect to $x$ by $n$ times, we get
  \[
  \partial_t \partial^n_x F -\left(m+\frac{1}{2}\right)\partial^{n+1}_x F +\frac{1}{2}\partial^n_x \left((\partial_x F)^2\right) +\partial^n_x \left(\frac{F}{x}\right)=0\,.
  \]
  Let $Q=\partial^n_x F$. Then
  \begin{equation}\label{eq:Q}
      \partial_t Q  -\left(m+\frac{1}{2}-\partial_x F\right) \partial_x Q =f(x,t)\,,
  \end{equation}
  where the source term $f$ is 
  \begin{multline*}
      f(x,t)= -n (\partial^2_x F) Q - \frac{Q}{x} - \frac{1}{2}\sum_{k=2}^{n-2} \frac{n! (\partial^{k+1}_x F)  (\partial^{n+1-k}_x F)}{k!(n-k)!} 
      - \sum_{k=0}^{n-1} \frac{(-1)^{n-k}n! (\partial^k_x F)}{k! \ x^{n-k+1}}\,.
  \end{multline*}
 Recall that \eqref{eq:Q} has the same characteristics $X(x,t)$ as in Proposition \ref{prop: F smooth}
 \[
 \dot X = -\left(m+\frac{1}{2}\right) + \partial_x F(X(t),t)\,,
 \]
  and \eqref{eq:X dot} holds.
  Thanks to Lemma \ref{lem: finer on second derivative} and \eqref{another modulus}, for fixed $T>0$, we are able to find a modulus of continuity $\omega:(0,\infty) \to [0,\infty)$ with $\lim_{r\to 0^+} \omega(r)=0$ such that 
  \[
  \left| x \partial^2_x F(x,t) \right| \leq \omega(r) \quad \text{ for all } (x,t) \in (0,r]\times [0,T]\,.
  \]
  Let $R(t)=Q(X(t),t)$  and fix $r>0$.
  As $X(t)$ reaches 0 in finite time, we can pick $s_0 \geq 0$ to be the smallest constant such that $X(s_0) \leq r$.
  Surely, $s_0=0$ in case $X(0)=x \leq r$.
    Without loss of generality, we assume that for some $t \geq s_0$, $X(t)>0$, and
    \begin{equation}  \label{eq:definitionM}
      M \defeq X(t)^{n-1}|R(t)| = \max_{s\in [s_0,t]} X(s)^{n-1}|R(s)|>0\,.
    \end{equation}

    \emph{Step 2.}
      It is our goal to bound $X(t)^{n-1} R(t)$ uniformly in $x$.
     Again, without loss of generality, we may assume that $R(s)$ does not change sign for $s\in (s_0,t]$ (otherwise, change $s_0$ to be a bigger constant such that $R(s_0)=0$ and $R(s)$ does not change sign for $s\in (s_0,t]$).
      Let us note right away that $-\frac{Q}{X}=-\frac{R}{X}$ is a good term and needs not to be controlled.
    Indeed, if $R>0$ in $(s_0,t)$, then $-\frac{R}{X} \leq 0$ there, and so
  \begin{equation} \label{ine:R-est}
    \begin{aligned}
  \abs{R(t) }&= R(t)=R(s_0) + \int_{s_0}^t f(X(s),s)\,ds\\
  &\leq R(s_0) + \int_{s_0}^t -n(\partial^2_x F)R\,ds\\
  &\quad +\int_{s_0}^t \left(- \frac{1}{2}\sum_{k=2}^{n-2} \frac{n! (\partial^{k+1}_x F)  (\partial^{n+1-k}_x F)}{k!(n-k)!} 
      - \sum_{k=0}^{n-1} \frac{(-1)^{n-k}n! (\partial^k_x F)}{k!X^{n-k+1}}\right) \,ds\,.
    \end{aligned}
  \end{equation}
  A similar claim holds in case $R<0$ in $(s_0,t)$.
  A key point that we need here in order to bound the above complicated sum is that, for $i \geq 2$, by~\eqref{eq:X dot}
  \begin{equation}\label{eq: bound of power X}
      \int_{s_0}^t \frac{1}{X(s)^i}\,ds \leq \int_{s_0}^t \frac{-2\dot X(s)}{X(s)^i}\,ds
      \leq \frac{2}{i-1}\left(\frac{1} {X(t)^{i-1}} - \frac{1}{r^{i-1}} \right)\,.
  \end{equation}
  This, together with the induction hypothesis, gives us that
  \begin{equation}  \label{ine:Rbound1}
  X(t)^{n-1}\left| \int_{s_0}^t \left(- \frac{1}{2}\sum_{k=2}^{n-2} \frac{n! (\partial^{k+1}_x F)  (\partial^{n+1-k}_x F)}{k!(n-k)!} 
      - \sum_{k=0}^{n-1} \frac{(-1)^{n-k}n! (\partial^k_x F)}{k!X^{n-k+1}}\right) \,ds \right| \leq C\,.
  \end{equation}
  Let us next bound the remaining term containing $R$.
  As $-\omega(r) \leq x \partial^2_x F \leq 0$ in $(0,r] \times [0,T]$, one has
  \begin{equation} \label{ine:Rbound2}
    \begin{aligned} 
    n \left|\int_{s_0}^t (\partial^2_x F) R\,ds \right| &\leq \int_{s_0}^t \frac{n\omega(r)M}{X(s)^{n}}\,ds \leq \frac{2n\omega(r)M}{n-1}\left(\frac{1}{X(t)^{n-1}} - \frac{1}{r^{n-1}}\right)\\
    &\leq \frac{3 \omega(r)M}{X(t)^{n-1}} \leq \frac{M}{2 X(t)^{n-1}}
    \end{aligned}
  \end{equation}
  for $r>0$ small enough.
  
  Combining~\eqref{eq:definitionM},~\eqref{ine:R-est},~\eqref{ine:Rbound1} and~\eqref{ine:Rbound2}, we deduce that
  \[
  M \leq C +\frac{M}{2}\,,
  \]
which yields that $M \leq 2C$. By definition of $M$, $X(t)$ and $R(t)$, we reach the desired result. 
\end{proof}

We are now ready to prove Proposition \ref{prop:F changing signs derivatives} by induction.
Our idea here is to use the maximum principle for $x^{k-1} \partial^k_x F$ for $k \geq 3$ in the inductive argument.
However, as the behavior of $x^{k-1} \partial^k_x F$ is unclear as $x \to 0^+$,  we need to use localizations around characteristics to take care of this issue.

\begin{proof}[{\bf Proof of Proposition \ref{prop:F changing signs derivatives}}]
  Let us show that $(-1)^{n+1} \partial^n_xF \geq 0$ in $(0,\infty)^2$ by induction.
  By Theorem~\ref{thm:regularity-m-less-than-quarter}, this is true for $n=2$ already.
  Assume that this is true for all $n \leq k-1$ for some $k \geq 3$. 
  We now show that this is true for $n=k$.
  Let us just deal with the case that $k$ is even as the other case can be done analogously.
  
  \emph{Step 1.}
  Differentiate \eqref{e:main} with respect to $x$ by $k$ times, we get
  \begin{equation}\label{eq: diff k}
  \partial_t \partial^k_x F -\left(m+\frac{1}{2}\right)\partial^{k+1}_x F +\frac{1}{2}\partial^k_x \left((\partial_x F)^2\right) +\partial^k_x \left(\frac{F}{x}\right)=0\,.
  \end{equation}
 Let $W(x,t)= x^{k-1} \partial^k_x F$, and we aim at deriving a PDE for $W$.
 As always, the last term on the left hand side above is not so easy to deal with. 
 The following is a new insight to handle this term thanks to Lemma \ref{lem:bound of higher derivatives},
 \begin{align*}
    x^{k-1}\partial^k_x \left(\frac{F}{x}\right) 
    &=x^{k-1}\partial^k_x \left(\int_0^1 \partial_x F(rx,t)\,dr\right)  \\
    &= x^{k-1}\int_0^1 r^k \partial^{k+1}_x F(rx,t)\,dr
    =\frac{1}{x^2}\int_0^x z^k \partial^{k+1}_x F(z,t)\,dz\\
    &=\frac{W(x,t)}{x} - \frac{k}{x^2}\int_0^x W(z,t)\,dz\,.
 \end{align*}
 We used integration by parts in the last equality above.
 Multiply \eqref{eq: diff k} by $x^{k-1}$ and use the above identity, we arrive at
 \begin{multline}\label{eq: W}
     \partial_t W -\left(m+\frac{1}{2}-\partial_x F \right) \left( \partial_x W - (k-1)\frac{W}{x} \right) + \frac{W}{x} - \frac{k}{x^2}\int_0^x W(z,t)\,dz\\
     = - k (\partial^2_x F) W - x^{k-1} \sum_{i=2}^{k-2}\frac{k! (\partial^{i+1}_x F)(\partial^{k+1-i}_x F)}{i!(k-i)!}\,.
 \end{multline}
 Again, this equation has the same characteristics $X(x,t)$ as in Proposition~\ref{prop: F smooth},
 \[
 \dot X = -\left(m+\frac{1}{2}\right) + \partial_x F(X(t),t)
 \]
  and \eqref{eq:X dot} holds. 
  This clear localization of characteristics is very important.
  
  \emph{Step 2.} 
  We now need to show that $W \leq 0$ in $(0,\infty)^2$.
  Assume by contradiction that there exists $(x_0,T) \in (0,\infty)^2$ such that $W(x_0,T)>0$.
  Of course, $x_0=X(z,T)$ for some $z>x_0$.
  
  For the initial condition of $W$, it is not hard to see that $W(0,0)=0$ and $W(x,0) \leq 0$ for $x\in [0,\infty)$.
  Choose $z_1,z_2$ very close to $z$ such that $z_1<z<z_2$, and define a new initial condition $\widetilde W(\cdot,0)$, which is smooth on $[0,\infty)$, such that
  \[
  \begin{cases}
  \widetilde W(x,0) = W(x,0) \quad &\text{ for } x \in [z_1,z_2]\,,\\
  \widetilde W(x,0) \leq W(x,0) \quad &\text{ for } x \notin [z_1,z_2]\,.
  \end{cases}
  \]
  Let $\widetilde W$ be the solution to \eqref{eq: W} corresponding to this new initial condition $\widetilde W(\cdot,0)$.
  Because of the locality of the characteristics,
  we see that $\widetilde W(x_0,T)=W(x_0,T)$.
  In fact, we can choose $\widetilde W(\cdot,0)$ to be as negative as we wish outside of $[z_1,z_2]$.
  For our purpose, we choose $z_1,z_2$, and $\widetilde W(\cdot,0)$ so that
  \begin{equation}\label{eq: W restriction}
      \widetilde W(x,t) < \widetilde W(X(z,t),t) \quad \text{ for all } x \in \left(0,\frac{x_0}{2}\right] \cup [z+1,\infty), \, t \in [0,T]\,.
  \end{equation}

  Now, slightly abusing the notations, let us assume that $W$ satisfies \eqref{eq: W restriction} as well (in other words, write $W$ in place of $\widetilde W$ for simplicity).
  For each $t \in [0,T]$, by~\eqref{eq: W restriction}, there exists  $x_t \in \left(\frac{x_0}{2},z+1\right)$ so that
  \[
  \xi(t)\defeq \max_{x \in [0,\infty)} W(x,t) = W(x_t,t) \,.
  \]

  We use the maximum principle in \eqref{eq: W} to get an estimate for $\xi$.
  Notice that, as $k$ is even, $(\partial^{i+1}_x F)(\partial^{k+1-i}_x F)\geq 0$ for $2 \leq i \leq k-2$ always by the induction hypothesis.
  At $(x_t,t)$, we have $\partial_x W(x_t,t)=0$, and
  \[
  \frac{1}{x_t}\int_0^{x_t} W(z,t)\,dz \leq W(x_t,t)\,.
  \]
  Therefore,
  \[
  \xi'(t) + \frac{\xi(t)}{x_t} \left ((k-1)\left(m - \frac{1}{2}-\partial_x F \right) + k x_t \partial^2_x F(x_t,t) \right) \leq 0\,.
  \]
  Note that $x_t \in \left(\frac{x_0}{2},z+1\right)$, and
  \[
  \left |(k-1)\left(m - \frac{1}{2}-\partial_x F \right) + k x_t \partial^2_x F(x_t,t) \right| \leq 2k\,.
  \]
  As $\xi(0)\leq 0$, by the usual differential inequality, we get that $\xi(t) \leq 0$ for all $t\in [0,T]$.
  In particular, $0 \geq \xi(T) \geq W(x_0,T) > 0$, which is absurd.
  The proof is complete.
  \end{proof}

\begin{proof}[{\bf Proof of Theorem \ref{thm: existence of solution to C-F}}]
  The result follows immediately by combining Propositions \ref{prop: F smooth},  \ref{prop:F changing signs derivatives} and Theorem~\ref{t:bernsteinRep}.
\end{proof}

\section{Equilibria}
In this section, we study the equilibria of equation \eqref{e:viscosity} in the case $0<m\leq 1$. 
At equilibrium, the equation reads
\begin{equation} \label{e:equilibrium}
   \frac{1}{2}(\partial_x F - m)(\partial_x F - m -1) + \frac{F}{x} - m =  0 \,.
\end{equation}
Let us emphasize again that we search for Lipschitz, sublinear viscosity solution $F$ which satisfies $0\leq F(x) \leq mx$ for $x\in [0,\infty)$.
\begin{lemma} \label{l:minimizers}
  Suppose $0<m <1$. 
  Let $F$ be a Lipschitz, sublinear viscosity solution to equation~\eqref{e:equilibrium} which satisfies $0\leq F(x) \leq mx$ for $x\in [0,\infty)$. 
  Then, there exists a constant $C>0$ so that all the local minimums of $F$ belong to $[0,C]$. 
\end{lemma}
\begin{proof}
  By contradiction, if there exists a sequence of local minimums $\{x_n\} \to \infty$ of $F$, then by the supersolution test, we have
  \begin{equation*}
    \frac{1}{2}m(m+1) + \frac{F(x_n)}{x_n} - m \geq 0\,.
  \end{equation*}
  This means, for $n\in \N$,
  \begin{equation*}
    \frac{F(x_n)}{x_n} \geq \frac{1}{2} m(1-m) >0\,,
  \end{equation*}
  which is a contradiction as $F(x_n)/x_n \to 0$ by the sublinearity assumption.
\end{proof}

\begin{proposition} \label{prop:nonExistenceSublinearSoln}
  Suppose $0<m <1$. 
  Then equation~\eqref{e:equilibrium} has no Lipschitz, sublinear viscosity solution $F$ which satisfies $0\leq F(x) \leq mx$ for $x\in [0,\infty)$.
\end{proposition}

\begin{proof}
  Suppose by contradiction that $F$ is a Lipschitz, sublinear solution to equation~\eqref{e:equilibrium} and $0\leq F(x) \leq mx$ for $x\in [0,\infty)$.
  By Lemma~\ref{l:minimizers}, there exists a $C>0$ so that $F(x)$ is monotone on $[C,\infty)$, i.e., for a.e. $x\in [C,\infty)$, either
  \begin{equation*}
    F'(x) \geq 0 \quad \text{ or } \quad F'(x) \leq 0 \,.
  \end{equation*}
Let us consider two cases in the following.

\smallskip

  \emph{Case 1.} $F'(x) \geq 0$ for a.e. $x \geq C$.  
  Since $F(x) \leq mx$, we have
  \begin{equation*}
    \frac{1}{2}( F'(x) - m)( F'(x) - m -1) = m - \frac{F(x)}{x} \geq 0 \,.
  \end{equation*}
  Thus, either $F'(x) \leq m$ or $F'(x) \geq m+1$. 
  We claim that $F'(x) \leq m$ for a.e. $x \geq C$ by changing $C$ to be a bigger value if needed.
  Indeed, assume otherwise, that this is not the case.
  Since $F(x) \leq mx$, we cannot have that $F'(x) \geq m+1$ for a.e. $x>C$.
  Then, we can find $x_2>x_1>C$ such that $F'(x_1) \leq m$, and $F'(x_2) \geq m+1$.
  Let $\phi(x)=(m+\frac{1}{2})x$ for $x\in [x_1,x_2]$ be a test function, and let $x_3 \in [x_1,x_2]$ be a minimum point of $F-\phi$ on $[x_1,x_2]$.
  As
  \[
  F'(x_1) \leq m < \phi'(x_1)=m+\frac{1}{2} = \phi'(x_2) < m+1 \leq F'(x_2)\,,
  \]
it is clear that $x_3 \neq x_1$ and $x_3 \neq x_2$.
In other words, $x_3 \in (x_1,x_2)$, and one is able to use the viscosity supersolution test to yield that
\[
0 \leq \frac{1}{2}\left(m+\frac{1}{2}-m\right)\left(m+\frac{1}{2}-m-1\right) +\frac{F(x_3)}{x_3}-m \leq -\frac{1}{8}\,,
\]
 which is absurd. Therefore,
  \begin{equation*}
    0 \leq F'(x) \leq m \quad \text{ for a.e. } x \geq C \,.
  \end{equation*}
  In particular, for a.e. $x\geq C$,
  \begin{equation*}
    \frac{1}{2}( F'(x) - m)(F'(x) - m - 1) \leq \frac{1}{2}( 0 - m) (0 - m -1) = \frac{1}{2} m(m+1) \,,
  \end{equation*}
  which implies 
  \begin{equation*}
    \frac{F(x)}{x} \geq m - \frac{1}{2} m(m+1) = \frac{1}{2}m(1-m)>0\,.
  \end{equation*}
  But this means that $F$ is not sublinear.

  \smallskip

  \emph{Case 2.} $F'(x) \leq 0$ for a.e. $x \geq C$.  
   Then $F$ is decreasing on $[C,\infty)$ and there exists $\alpha \geq 0$ such that $\alpha=\lim_{x\to \infty} F(x)$. 
  Consequently, 
  \[
  \lim_{x \to \infty} \left( \frac{1}{2}(F'(x)-m)(F'(x)-m-1) -m\right) =0\,.
  \]
On the other hand, as $F \geq 0$ always, we can find a sequence $\{y_n\} \to \infty$ such that $F'(y_n) \to 0$.
Let $x=y_n$ in the above and let $n \to \infty$ to deduce that
  \begin{align*}
    0 = \frac{1}{2}(0-m)(0-m-1) -m =\frac{1}{2}m(m-1)<0 \,,
  \end{align*}
which is absurd.

  \smallskip

  Therefore, in all cases, we are led to contradictions.
  The proof is complete.
\end{proof}

\begin{proposition} \label{prop:equilibrium}
  Let $m = 1$. Then equation~\eqref{e:equilibrium} admits a Lipschitz, sublinear viscosity solution $F$ which satisfies $0\leq F(x) \leq mx$ for $x\in [0,\infty)$.
\end{proposition}

\begin{proof}
  Let $G = 1- \partial_x F$. Then the equilibrium equation reads as
  \begin{equation}\label{eq:G}
    \frac{1}{2} G(G+1) - \frac{1}{x} \int_0^x G = 0 \,.
  \end{equation}
  This is the same equation studied in the work of Degond, Liu and Pego~\cite[Section 5]{DegondLiuEA17}, of which the solution must satisfy the transcendental equation
  \begin{equation}\label{eq:G-1}
    \frac{G(x)}{(1 - G(x))^3} = Cx\,, 
  \end{equation}
  for some constant $C>0$.
  Let us recall a quick proof of \eqref{eq:G-1}.
  Multiply \eqref{eq:G} by $x$, then differentiate the result with respect to $x$ to imply
  \[
  \frac{1}{2}G(G+1) + \frac{1}{2}x(2G \partial_x G + \partial_x G) - G=0\,,
  \]
  which means that
  \[
  \frac{1}{x}=\frac{3 \partial_x G}{1-G} + \frac{\partial_x G}{G}\,.
  \]
  Integrate the above to yield \eqref{eq:G-1}.
  Therefore, we can pick $C=1$ in \eqref{eq:G-1} and $G$ to be  a Bernstein function taking the form
  \begin{equation}\label{eq:formula G}
    G(x) = \int_0^\infty (1 - e^{-sx})\gamma(s) e^{-4s/27} \, ds\, , 
  \end{equation}
  where 
  \begin{equation*}
    \int_0^\infty \gamma(s) e^{-4s/27} \, ds = 1\,.
  \end{equation*}
  See \cite[Section 5]{DegondLiuEA17} for further details on the derivation of \eqref{eq:formula G}.
  This implies that 
  \begin{equation}
    \partial_x F (x)= 1 - \int_0^\infty (1 - e^{-sx})\gamma(s) e^{-4s/27} \, ds \geq 0\, ,
  \end{equation}
  and that
  \begin{equation*}
    \lim_{x\to \infty} \frac{F(x)}{x} =  0 \,.
  \end{equation*}
  Furthermore, by successive differentiations, we can also see that $\partial_x F$ is completely monotone, that is, $(-1)^{n+1}\partial^n_x F \geq 0$ for all $n\in \N$, which means that $F$ is a Bernstein function.
\end{proof}

\begin{remark} \label{rem:family}
From the above proposition, it is actually not hard to see that, for $m=1$, equation~\eqref{e:equilibrium} admits a family of Lipschitz, sublinear viscosity solution $\{F_\lambda\}_{\lambda>0}$ which satisfies $0\leq F_\lambda(x) \leq x$ for $x\in [0,\infty)$.
Indeed, take $F$ as in the above proof, and denote by
\[
F_\lambda(x) = \lambda F\left(\frac{x}{\lambda}\right) \quad \text{ for all } x \in [0,\infty)\,.
\]
Then,
\[
G_\lambda(x) = 1 - \partial_x F_\lambda(x) = 1 - \partial_x F \left(\frac{x}{\lambda}\right) = G \left(\frac{x}{\lambda}\right)\,,
\]
which means that
\[
\frac{G_\lambda(x)}{(1-G_\lambda(x))^3}= \frac{x}{\lambda}\,.
\]
This implies that \eqref{eq:G-1} is satisfied with $C=\frac{1}{\lambda}$.
Hence, $F_\lambda$ is a solution to \eqref{e:equilibrium} for each $\lambda>0$.

The existence of this family of solutions $\{F_\lambda\}_{\lambda>0}$ to \eqref{e:equilibrium} makes the study of large time behavior of the viscosity solution to \eqref{e:main} for $m=1$ quite difficult.
\end{remark}

\section{Large time behavior for \texorpdfstring{$0<m < 1$}{ m <1 }}
In this section, we study the large time behavior of the viscosity solution to equation~\eqref{e:main} for $0<m<1$.
Our goal is to prove Theorem \ref{t:largeTimeLessThan1}.

From Proposition~\ref{prop:nonExistenceSublinearSoln}, one cannot expect a sublinear equilibrium, that is, a Lipschitz sublinear solution to \eqref{e:equilibrium}. 
However, it is very interesting that the solution to equation~\eqref{e:main} still converges to the  linear function $mx$ locally uniformly as $t\to \infty$. 
This implies that, even if we have a mass-conserving solution at all time, the sizes of particles decrease until they become dust at time infinity.

To prove the theorem, we need the following results.

\begin{lemma}\label{l:characterStationaryLessThan1}
  Let $\bar F$ be a viscosity supersolution to  equation \eqref{e:equilibrium} that satisfies the following
  \begin{equation} \label{con:concave}
    \begin{cases}
      \text{ $\bar F$ is concave} \,, \\
      \liminf_{x\to\infty} \frac{\bar F(x)}{x} > 0 \,,  \\
      0\leq \bar F(x) \leq mx \,.
  \end{cases}
  \end{equation}

  Then, $\bar F(x) = mx$.
\end{lemma}

\begin{proof}
  First, observe that $x\mapsto \partial_x \bar F(x)$ is decreasing whenever $\partial_x \bar F(x)$ is defined. By the requirement that 
  \begin{equation*}
    \liminf_{x\to\infty} \frac{\bar F(x)}{x} >0 \,,
  \end{equation*}
  we have that $\partial_x \bar F(x) \geq 0$.
  As $\bar F$ is differentiable almost everywhere, pick $\{x_n\} \to \infty$  so that $F$ is differentiable at $x_n$ for all $n\in \N$. 
   Denote
  \begin{equation*}
    0 < \alpha \defeq \lim_{n\to\infty} \partial_x \bar F(x_n) = \lim_{x\to\infty} \frac{\bar F(x_n)}{x_n} \leq m \,.
  \end{equation*}
  Thus, letting $x_n\to\infty$ in the equation~\eqref{e:equilibrium}, we get 
  \begin{equation*}
    0 \leq \frac{1}{2} (\alpha - m)(\alpha - m+1) \leq 0 \,.
  \end{equation*}
  Therefore, it is necessary that $ \alpha = m$ and $\bar F(x) = mx$ for all $x\in [0,\infty)$.
\end{proof}

We immediately have the following consequence.
\begin{corollary}
  Let $\bar F$ be a viscosity solution to equation~\eqref{e:equilibrium} satisfying \eqref{con:concave}. Then $\bar F(x) = mx$ for $x\in [0,\infty)$.
\end{corollary}

\begin{lemma} \label{l:lowerBoundLessThan1}
  Let $F$ be the Lipschitz, sublinear viscosity solution to equation~\eqref{e:main}. Then, locally uniformly for $x\in [0,\infty)$,
  \begin{equation}\label{ine:lowerBoundLessThan1}
    \liminf_{t\to\infty} F(x,t) \geq \frac{1}{4} m(1-m) x \,.
  \end{equation}
\end{lemma}
\begin{proof}
  We construct a sublinear subsolution to the equation~\eqref{e:main} so that the inequality~\eqref{ine:lowerBoundLessThan1} holds. 
  Define, for $(x,t) \in [0,\infty)^2$,
  \begin{equation*}
    \varphi(x,t) \defeq \min \left\{ \frac{1}{4} m(1- m)x, \frac{1}{4} m(1-m)t  \right\}\,.
  \end{equation*}
  To see that $\varphi$ is a subsolution to~\eqref{e:main}, we first note that $\frac{1}{4} m(1- m)x$ is a subsolution.
  Furthermore, 
  \begin{equation*}
    \varphi(x,t) = \begin{cases}
        \frac{1}{4} m(1-m) x \,, \quad  x <t \,, \\
        \frac{1}{4} m(1- m) t \,, \quad x\geq t  \,.
    \end{cases}
  \end{equation*}
  So, for $x > t$, 
  \begin{align*}
    & \partial_t \varphi + \frac{1}{2} ( \partial_x \varphi - m)(\partial_x \varphi - m - 1) + \frac{\varphi}{x} - m  \\
    \leq\, & \frac{1}{4} m(1-m) + \frac{1}{2} m(m-1) + \frac{1}{4} m(1-m) = 0 \,.
  \end{align*}
  Since equation~\eqref{e:main} has a convex Hamiltonian, minimum of two subsolutions is a subsolution (see Tran~\cite[Chapter 2]{Tran19} and the references therein). 
  Note that this property is not true for general Hamiltonians.
  
  By the comparison principle, we have that $F \geq \varphi$. 
  Letting $t \to \infty$, we obtain \eqref{ine:lowerBoundLessThan1} locally uniformly for $x\in [0,\infty)$.
\end{proof}

\begin{proof}[{\bf Proof of Theorem~\ref{t:largeTimeLessThan1}}]
  By Lemma~\ref{l:lowerBoundLessThan1}, locally uniformly for $x\in (0,\infty)$,
  \begin{equation*}
    m \geq \liminf_{t\to \infty} \frac{F(x,t)}{x}  \geq \frac{1}{4} m(1 - m) >0 \,.
  \end{equation*}
  Let 
  \begin{equation*}
    G(x) \defeq \liminf_{t\to \infty} F(x,t) \quad \text{ for all } x \in [0,\infty) \,.
  \end{equation*}
  This function is well-defined since $F$ is globally Lipschitz on $[0,\infty)^2$ and $0\leq F(x,t) \leq mx$.
  By stability of viscosity solutions, $G$ is a supersolution to equation
  \begin{equation*}
    \frac{1}{2} (\partial_x G - m) ( \partial_x G - m - 1) + \frac{G}{x} - m \geq 0 
  \end{equation*}
  in $(0,\infty)$.
  As $x \mapsto F(x,t)$ is concave for every $t \geq 0$, $G$ is concave. Moreover, $0 \leq G \leq mx$ and 
  \begin{equation*}
    G(x) \geq \frac{1}{4} m(1- m) x \quad \text{ for all } x\in [0,\infty) \,.
  \end{equation*}
  By Lemma~\ref{l:characterStationaryLessThan1}, $G(x) = mx$ for $x\in [0,\infty)$. 
  We use this and the fact that $F(x,t) \leq mx$ for all $(x,t) \in [0,\infty)^2$ to conclude that, locally uniformly for $x\in [0,\infty)$,
  \begin{equation*}
    \lim_{t\to\infty} F (x,t) = G(x) = mx \,,
  \end{equation*}
  as desired.
\end{proof}

It is worth noting that \eqref{ine:lowerBoundLessThan1} is only useful for $0<m<1$, and is meaningless when $m=1$.
Large time behavior of $F$ in case $m=1$ remains an open problem.

\appendix
\section{Bernstein functions and transform} \label{A:bernstein}
In this appendix, we record a representation theorem of Bernstein functions, which is important for the inference of existence of solutions to equation~\eqref{eq:CF} from smooth solution of equation~\eqref{e:main}. 
\begin{definition}
    A function $f: (0,\infty) \to [0,\infty)$ is a Bernstein function if $f \in C^\infty((0,\infty))$ and, for $n\in \N$,
    \begin{equation*}
        (-1)^{n+1} \frac{d^n}{dx^n} f \geq 0 \,.
    \end{equation*}
\end{definition}

\begin{theorem}\label{t:bernsteinRep}
    A function $f:(0,\infty) \to [0,\infty)$ is a Bernstein function if and only if it can be written as
    \begin{equation}\label{eq: rep f}
        f(x) = a_0 x + a_\infty + \int_{(0,\infty)} (1- e^{-sx}) \mu (ds), \qquad x\in (0,\infty) \,,
    \end{equation}
    where $a_0, a_\infty \geq 0$ and $\mu$ is a measure such that
    \begin{equation*}
        \int_{(0,\infty)} \min\set{ 1, s} \, \mu(ds) <\infty \,.
    \end{equation*}
\end{theorem}
In other words, a Bernstein function is a Bernstein transform on the extended real line $[0,\infty]$.
The proof of this theorem and more beautiful properties of Bernstein functions and transform could be found in the book by Schilling, Song, and Vondra\v{c}ek~\cite{SchillingSong12}.

\medskip

Next, consider $f:[0,\infty) \to [0,\infty)$ which is a Bernstein function such that $f(0)=0$ and $f$ is sublinear.  
By Theorem \ref{t:bernsteinRep}, $f$ has the representation formula \eqref{eq: rep f}.
Firstly, let $x\to 0^+$ to get that
\[
a_\infty = \lim_{x \to 0^+} f(x) =0\,.
\]
Secondly, divide \eqref{eq: rep f} by $x$, let $x \to \infty$ and use the sublinearity of $F$ to yield further that
\[
a_0=\lim_{x \to \infty} \frac{f(x)}{x}=0\,.
\]
Thus, under two additional conditions that $f(0)=0$ and $f$ is sublinear, we get that $a_0=a_\infty=0$, and therefore,
\[
f(x) = \int_{(0,\infty)} (1- e^{-sx}) \mu (ds), \qquad x\in (0,\infty) \,.
\]

\section*{Acknowledgement}
TSV thanks Bob Pego and Philippe Lauren\c{c}ot for the introduction to the C-F as well as many helpful discussions. 
Proposition~\ref{prop:equilibrium} and Remark~\ref{rem:family} are based on an observation of Lauren\c{c}ot through private communications.
TSV also thanks Hausdorff Research Institute for Mathematics
(Bonn), through the Junior Trimester Program on Kinetic Theory, for their welcoming environment during the period that he learned most about the questions about C-F addressed in this paper.
HT thanks Tuoc Phan for a discussion on regularity of degenerate parabolic equations, and Hiroyoshi Mitake for some useful suggestions.
We also would like to thank the anonymous referee very much for carefully reading our manuscript and giving very helpful comments to improve its presentation.
HT is supported in part by NSF grant DMS-1664424 and NSF CAREER grant DMS-1843320. 

\printbibliography

\end{document}